\documentclass[12pt]{amsart}
\usepackage{amscd,amssymb}
\usepackage{amsthm,amsmath,amssymb}
\usepackage[matrix,arrow,curve]{xy}
\usepackage{mathrsfs}
\usepackage{wasysym}
\usepackage[dvipsnames]{xcolor}

\newtheorem{theorem}[equation]{Theorem}

\newtheorem{proposition}[equation]{Proposition}
\newtheorem{lemma}[equation]{Lemma}
\newtheorem{corollary}[equation]{Corollary}
\newtheorem{conjecture}[equation]{Conjecture}
\newtheorem{question}[equation]{Question}

\theoremstyle{definition}
\newtheorem{example}[equation]{Example}

\theoremstyle{remark}
\newtheorem{remark}[equation]{Remark}

\makeatletter\@addtoreset{equation}{section} \makeatother

\def\wt {\mathrm{wt}}
\def\P {\mathbb{P}}
\def\Q {\mathbb{Q}}

\author{Ivan Cheltsov, Victor Przyjalkowski, Constantin Shramov}

\title{Which quartic double solids are rational?}

\address{\emph{Ivan Cheltsov}
\newline
\textnormal{School of Mathematics, The University of Edinburgh,  Edinburgh EH9 3JZ, UK.}
\newline
\textnormal{National Research University Higher School of Economics, Laboratory of Algebraic Geometry, 6 Usacheva str., Moscow, 119048, Russia.}
\newline
\textnormal{\texttt{I.Cheltsov@ed.ac.uk}}}

\address{\emph{Victor Przyjalkowski}
\newline
\textnormal{Steklov Institute of Mathematics, 8 Gubkina street, Moscow 119991, Russia.}
\newline
\textnormal{National Research University Higher School of Economics, Russian Federation,
Laboratory of Mirror Symmetry, NRU HSE, 6 Usacheva str., Moscow, Russia, 119048
}
\newline
\textnormal{\texttt{victorprz@mi.ras.ru, victorprz@gmail.com}}}

\address{\emph{Constantin Shramov}
\newline
\textnormal{Steklov Institute of Mathematics, 8 Gubkina street, Moscow 119991, Russia.}
\newline
\textnormal{National Research University Higher School of Economics, Laboratory of Algebraic Geometry, 6 Usacheva str., Moscow, 119048, Russia.}
\newline
\textnormal{\texttt{costya.shramov@gmail.com}}}

\begin{document}

\begin{abstract}
We study the rationality problem for nodal quartic double solids.
In particular, we prove that nodal quartic double solids with at most six singular points are irrational,
and nodal quartic double solids with at least eleven singular points are rational.
\end{abstract}

\sloppy

\maketitle

\section{Introduction}
\label{section:into}

In this paper, we study double covers of $\mathbb{P}^3$ branched over nodal quartic surfaces.
These Fano threefolds are known as \emph{quartic double solids}.
It is well-known that smooth threefolds of this type are irrational.
This was proved by Tihomirov (see~\mbox{\cite[Theorem~5]{Tihomirov3}})
and Voisin (see \cite[Corollary~4.7(b)]{Voisin88}).
The same result was proved by Beauville
in~\mbox{\cite[Exemple~4.10.4]{Beauville}}
for the case of quartic double solids with one ordinary double singular point (node),
by Debarre in \cite{Debarre}  for the case of up to four nodes and also for five nodes subject to generality conditions,
and by Varley in \cite[Theorem~2]{Varley} for double covers of $\mathbb{P}^3$ branched over special quartic surfaces with six nodes
(so-called Weddle quartic surfaces).
All these results were proved using the theory of intermediate Jacobians introduced by Clemens and Griffiths in \cite{ClemensGriffiths}.
In \cite[\S8~and~\S9]{Clemens},
Clemens studied intermediate Jacobians of
resolutions of singularities for nodal quartic double solids
with at most six nodes in general position.

Another approach to irrationality of  nodal quartic double solids was introduced by Artin and Mumford in \cite{ArtinMumford}.
They constructed an example of a quartic double solid with ten nodes whose resolution of singularities
has non-trivial torsion in the third integral cohomology group, and thus the solid is not stably rational.
Recently, Voisin used this example together
with her new approach via Chow groups to prove the following
result.

\begin{theorem}[{\cite[Theorem~0.1]{Voisin}}]
\label{theorem:Voisin}
For any integer $k=0,\ldots,7$, a very general nodal quartic double
solid with $k$ nodes is not stably rational.
\end{theorem}

In spite of its strength, Theorem~\ref{theorem:Voisin} is not easy to
apply to particular varieties.
This is due to non-explicit generality condition involved,
which is a common feature of many results based on degeneration method
and universal triviality of Chow groups, see e.g. the survey~\cite{Pirutka} and references therein.
In any case, it is natural to ask what one can say about birational geometry of an arbitrary
variety in our family, not just a general one.
The main goal of this paper is to get rid of the generality condition
(at the cost of weakening the assertion and allowing fewer singular
points on a quartic double solid).
We use some explicit birational constructions for conic bundles
together with a standard approach based on intermediate Jacobian theory
to prove the following result.

\begin{theorem}
\label{theorem:CPS}
A nodal quartic double solid with at most six nodes is irrational.
\end{theorem}

Recall that a nodal quartic surface in $\mathbb{P}^3$ can have at most $16$ nodes,
so that this is also the maximal number of nodes on a quartic double solid.
Moreover, Prokhorov proved that every nodal quartic double solid with $15$ or $16$ nodes is rational
(see \cite[Theorem~8.1]{Prokhorov} and \cite[Theorem~7.1]{Prokhorov}, respectively).
We use his approach to study quartic double solids with many nodes.
In particular, we prove the following result.

\begin{theorem}
\label{theorem:CPS-11}
A nodal quartic double solid with at least eleven nodes is rational.
\end{theorem}

In fact, we prove a stronger assertion.
To describe it, let us recall that a variety is said to be $\mathbb{Q}$-factorial if every Weil divisor on it is $\mathbb{Q}$-Cartier.
For a nodal variety, this condition is equivalent to the coincidence of Weil and Cartier divisors.

\begin{example}
\label{example:cubic-threefold}
Let $S$ be a nodal quartic surface in $\mathbb{P}^3$ with homogeneous co\-or\-di\-na\-tes~$x,y,z,t$ such that $S$ is given by the equation $g^2=4xh$,
where $g$ and $h$ are some quadric and cubic forms, respectively.
Then $S$  is singular at the points given by~\mbox{$x=g=h=0$}.
Since we assume that $S$ is nodal, this system of equations gives
exactly six po\-ints~$P_1,\ldots,P_6$.
Let~\mbox{$\tau\colon X\to\mathbb{P}^3$} be a double cover branched over $S$.
Then there exists a commutative diagram
$$
\xymatrix{
X^\prime\ar@{->}[rr]^{\alpha}\ar@{->}[d]_{\beta}&&V_3\ar@{-->}[d]^{\gamma}\\
X\ar@{->}[rr]^{\tau}&&\mathbb{P}^{3}.}
$$
Here $V_3$ is a smooth cubic threefold in $\mathbb{P}^4$
with homogeneous coordinates $x,y,z,t,w$
such that it is given by equation
$$
w^2x+wg+h=0,
$$
the map $\gamma$ is a linear projection from the point $P=[0:0:0:0:1]$, the map $\alpha$ is the blow up of this point,
and $\beta$ is the contraction of the proper transforms of six lines on $V_3$
that pass through $P$ to the singular points of~$X$.
Then the image on $X$ of the $\alpha$-exceptional surface is not a $\mathbb{Q}$-Cartier divisor.
In particular, $X$ is not $\mathbb{Q}$-factorial.
Note that $\beta$ is a small resolution of singularities of~$X$
at the points~\mbox{$P_1,\ldots,P_6$}.
Thus, $V_3$ is singular if and only if $P_1,\ldots,P_6$ are the only singular points of $X$.
On the other hand,  $V_3$ is irrational if and only if it is smooth (see \cite[Theorem~13.12]{ClemensGriffiths}).
Thus, $X$ is irrational if and only if it has exactly six singular points.
\end{example}

It turns out that Example~\ref{example:cubic-threefold} provides the only construction
of irrational nodal quartic double solids that are not $\mathbb{Q}$-factorial.
Namely, we prove the following result.

\begin{theorem}
\label{theorem:non-factorial}
A non-$\mathbb{Q}$-factorial nodal quartic double solid
is rational unless it has exactly six nodes
and is described by Example~\ref{example:cubic-threefold}.
\end{theorem}

Note that the $\mathbb{Q}$-factoriality of nodal quartic double solids can be easily verified using the following result of Clemens.

\begin{theorem}[{\cite[\S3]{Clemens}}]
\label{theorem:Clemens}
Let $X$ be a double cover of $\mathbb{P}^3$ branched over a nodal quartic surface $S$.
Then $X$ is $\mathbb{Q}$-factorial if and only if the nodes of $S$ impose independent linear conditions on quadrics in $\mathbb{P}^3$.
\end{theorem}

In particular, this theorem gives the following result that implies Theorem~\ref{theorem:CPS-11} by Theorem~\ref{theorem:non-factorial}.

\begin{corollary}
\label{corollary:factorial-double-solids}
Nodal quartic double solids with at least eleven nodes are not \mbox{$\mathbb{Q}$-factorial}.
\end{corollary}

On the other hand, nodal quartic double solids with at most five nodes are always $\mathbb{Q}$-factorial.
This is implied by the following result.

\begin{theorem}[{\cite{HongPark}, \cite[Theorem~6]{Cheltsov}}]
\label{theorem:Cheltsov}
A nodal quartic double solid with at most seven nodes is $\mathbb{Q}$-factorial unless it has exactly six nodes
and is described by Example~\ref{example:cubic-threefold}.
\end{theorem}

Thus, one can generalize our Theorem~\ref{theorem:CPS} as follows.

\begin{conjecture}
\label{conjecture:factorial}
Every $\mathbb{Q}$-factorial nodal quartic double solid is irrational.
\end{conjecture}

By Theorem~\ref{theorem:non-factorial}, this conjecture
gives a complete answer to the question in the title of this paper
in the case of nodal quartic double solids.
Here we show that it follows from Shokurov's famous \cite[Conjecture~10.3]{Shokurov},
see Corollary~\ref{corollary:Shokurov-factorial}.
Note that the intermediate Jacobians of resolutions of
singularities for nodal quartic double solids with more than six nodes
are sums of Jacobians of curves,
so that the methods of~\cite{ClemensGriffiths} are
not applicable to prove Conjecture~\ref{conjecture:factorial} in this case.

\medskip

The paper is organized as follows.
In Section~\ref{section:conic-bundles}, we review some well-known facts about conic bundles over rational surfaces
including their Prym varieties and results concerning irrationality of such threefolds.
In Section~\ref{section:double-solids}, we show how to birationally transform a nodal $\mathbb{Q}$-factorial singular quartic double solid
into a conic bundle and study the singularities of its degeneration curve.
In Section~\ref{section:standard}, we present an explicit birational transformation of the latter conic bundle to a standard one.
In Section~\ref{section:factorial}, we prove Theorem~\ref{theorem:CPS}
and show that \cite[Conjecture~10.3]{Shokurov} implies our Conjecture~\ref{conjecture:factorial}.
In Section~\ref{section:non-factorial}, we prove Theorem~\ref{theorem:non-factorial}.
In a sequel \cite{CheltsovPrzyjalkowskiShramov}, we apply Theorems~\ref{theorem:CPS} and \ref{theorem:CPS-11}
to nodal quartic double solids having an icosahedral symmetry.

\medskip
{\bf Acknowledgements.}
The authors are grateful to Olivier Debarre, Yuri Prokhorov, Vyacheslav Shokurov,
and Andrey Trepalin for useful discussions.
This work was started when Ivan Cheltsov and Constantin Shramov were visiting the Mathematisches Forschungsinstitut in Oberwolfach
under Research in Pairs program, and was completed when they were visiting Centro Internazionale per la Ricerca Matematica in Trento
under a similar program.
We want to thank these institutions for excellent working conditions,
and personally Marco Andreatta and Augusto Micheletti for hospitality.

Ivan Cheltsov was supported by the Russian Academic Excellence Project ``5-100'' and by the Royal Society grant No. IES\slash R1\slash 180205. Victor Przyjalkowski was supported by Laboratory of Mirror Symmetry NRU
HSE, RF government grant, ag. No. 14.641.31.0001.
Constantin Shramov was supported by the Russian Academic Excellence Project ``5-100'',
the grants  RFFI 15-01-02158, RFFI 15-01-02164, and the Foundation for the
Advancement of Theoretical Physics and Mathematics ``BASIS''.
The last two authors are also a Young Russian Mathematics award winners and would like to thank its sponsors and jury.

\medskip

{\bf Notation and conventions.} All varieties are assumed to be algebraic, projective and defined over $\mathbb{C}$.
By a \emph{node}, we mean an isolated ordinary double singular
point of a variety of arbitrary dimension.
A variety is called \emph{nodal} if its only singularities are nodes.
By a \emph{cusp}, we mean a plane curve singularity of type $\mathbb{A}_2$.
By a \emph{tacnode}, we mean a plane curve singularity of type $\mathbb{A}_3$.
By~$\mathbb{F}_n$ we denote a Hirzebruch
surface~\mbox{$\P\big(\mathcal{O}_{\P^1}\oplus\mathcal{O}_{\P^1}(n)\big)$}.
Given a birational morphism $\varphi\colon X\to Y$ and a linear
system $\mathcal{M}$ on $Y$,
by a proper transform (sometimes also called a homaloidal transform)
of $\mathcal{M}$ we mean the linear system generated by
divisors $\varphi^{-1}M$, where $M$ is a general divisor in $\mathcal{M}$; since a rational
map is a morphism in a complement to a closed subset of codimension~$2$,
we will also use this terminology in case when $\varphi$ is an arbitrary
birational map. If $\mathcal{M}$ is base point free and
$\varphi$ is an arbitrary rational map, then the proper transform
of $\mathcal{M}$ is defined as a composition of its pull-back
via a regularization of $\varphi$ and a proper transform via
the corresponding birational map.

\section{Conic bundles over rational surfaces}
\label{section:conic-bundles}

Let $\nu\colon V\to U$ be a conic bundle such that $V$ is a threefold, and let $U$ be a surface.
Recall that $\nu$ is said to be \emph{standard} if both $V$ and $U$ are smooth,
and the relative Picard group of $V$ over $U$ has rank $1$.
It is well-known that there exists a commutative diagram
$$
\xymatrix{
V^\prime\ar@{->}[d]_{\nu^\prime}\ar@{-->}[rr]^{\rho}&&{V}\ar@{->}[d]^{{\nu}}\\
U^\prime\ar@{-->}[rr]^{\varrho}&&{U}}
$$
such that $\rho$ and $\varrho$ are birational maps, and $\nu^\prime$
is a standard conic bundle (see,
for example,~\mbox{\cite[Theorem~1.13]{Sarkisov2}}).
Because of this, we assume here that  $\nu$ is already standard.
In particular, we assume that $V$ and $U$ are both smooth.

In this section, we discuss some obstructions for $V$ to be rational.
In particular, we also assume that $U$ is rational,
since otherwise irrationality of $V$ is easy to show.

Denote by $\Delta$ the degeneration curve of the conic bundle $\nu$.
Since $\nu$ is assumed to be standard, the curve $\Delta$ is nodal (see, for example, \cite[Corollary~1.11]{Sarkisov2}).
Restricting the conic bundle $\nu$ to~$\Delta$, taking the normalization of the resulting surface,
and considering the Stein factorization of the the induced morphism to $\Delta$,
we obtain a nodal curve $\Delta^\prime$ together with an involution $I$ on it such that
$$
\Delta^\prime/I\cong\Delta,
$$
the nodes of $\Delta^\prime$ are exactly the fixed points of $I$,
the involution $I$ does not interchange branches at these points, and nodes of $\Delta$ are exactly images of nodes of $\Delta'$. In particular,
the number of connected components of $\Delta^\prime$
is the same as that of $\Delta$.
Alternatively, one can construct $\Delta^\prime$ as a Hilbert scheme of lines in the fibers of $\nu$ over $\Delta$.

\begin{corollary}
\label{corollary:stability}
The curve $\Delta$ satisfies the following conditions:
\begin{itemize}
\item[(A)] \mbox{for every splitting $\Delta=\Delta_1\cup\Delta_2$,
the number $|\Delta_1\cap\Delta_2|$ is even};

\item[(B)] for every connected component $\Delta_1$ of the curve $\Delta$, one has $\Delta_1\not\cong\mathbb{P}^1$.
\end{itemize}
\end{corollary}

\begin{proof}
Assertion $(\mathrm{A})$ follows from the fact that a double cover of a smooth curve is ramified over an even number of points.
Assertion $(\mathrm{B})$ follows from the fact that the double co\-ver~$\Delta'\to \Delta$ is unramified over any smooth connected component of $\Delta$ since its nodes are images of nodes of $\Delta'$ and the fact that $\P^1$ does not have connected unramified double covers.
\end{proof}

In \cite{Shokurov}, Shokurov formulated the following conjecture.

\begin{conjecture}[{\cite[Conjecture~10.3]{Shokurov}}]
\label{conjecture:Shokurov}
If $|2K_U+\Delta|\ne\varnothing$, then $V$ is irrational.
\end{conjecture}

It follows from \cite[Th\'eor\`eme~4.9]{Beauville}
that this conjecture holds for $U=\mathbb{P}^2$.
In \cite[\S10]{Shokurov}, Shokurov proved that Conjecture~\ref{conjecture:Shokurov} holds also for $U=\mathbb{F}_n$.

\begin{remark}
\label{remark:Prym}
Let $\Gamma$ be a connected nodal curve.
Suppose that there exists a connected nodal curve $\Gamma^\prime$ together with an involution $\iota$ on it
such that $\Gamma^\prime/\iota\cong\Gamma$, the nodes of $\Gamma^\prime$ are exactly the fixed points of $\iota$,
and $\iota$ does not interchange branches at these points.
Then one can construct a principally polarized abelian variety $\mathrm{Prym}(\Gamma^\prime,\iota)$
known as the \emph{Prym variety} of the pair~\mbox{$(\Gamma^\prime, \iota)$}.
For details and basic properties of
$\mathrm{Prym}(\Gamma^\prime,\iota)$, see \cite[\S0]{Beauville}
or~\cite{Shokurov}.
\end{remark}

Consider $\mathrm{Prym}(\Delta^\prime,I)$. Its importance is due to the following result.

\begin{theorem}[{see \cite[Proposition~2.8]{Beauville} and the discussion before it}]
\label{theorem:Jac-Prym}
Let $\mathrm{J}(V)$ be the intermediate Jacobian of $V$.
Then $\mathrm{J}(V)\cong\mathrm{Prym}(\Delta^\prime,I)$ (as principally polarized abelian varieties).
\end{theorem}

The dualizing sheaf of the curve $\Delta$ is free
(see e.\,g.~\cite[Exercise~3.4(1)]{HarrisMorrison}).
Moreover, the linear system $|K_{\Delta}|$ is base point free by Corollary~\ref{corollary:stability}.
Hence, it gives the canonical morphism
$$
\kappa_\Delta\colon\Delta\to\mathbb{P}^{N},
$$
where $N=h^0(\mathcal{O}_{\Delta}(K_{\Delta}))-1$.
Note that $\kappa_\Delta$ may contract irreducible components of $\Delta$.
If~$\Delta$ is connected, then it is said to be
\begin{itemize}
\item \emph{hyperelliptic} if there is a morphism
$\Delta\to\mathbb{P}^1$ that has degree two
over a general point of~$\mathbb{P}^1$;
\item \emph{trigonal} if there is a  morphism
$\Delta\to \mathbb{P}^1$ that has degree three
over a general point of~$\mathbb{P}^1$;
\item \emph{quasitrigonal} if it is a hyperelliptic curve with two glued smooth points.
\end{itemize}

\begin{remark}
\label{remark:hypertrig}
Suppose that $\Delta$ is connected.
If the curve $\Delta$ is hyperelliptic,
then~\mbox{$\kappa_\Delta(\Delta)$} is a rational normal curve of degree $N$,
and the induced map $\Delta\to\kappa_\Delta(\Delta)$ has degree two
over a general point of $\kappa_\Delta(\Delta)$.
If the curve $\Delta$ is trigonal, then the curve $\kappa_\Delta(\Delta)$ has trisecants,
so that $\kappa_\Delta(\Delta)$ is not an intersection of quadrics.
If $\Delta$ is quasitrigonal, then the intersection of quadrics passing through $\kappa_\Delta(\Delta)$
is a cone over a rational normal curve of degree $N-1$. In particular, in this case $\kappa_\Delta(\Delta)$ is not an intersection
of quadrics as well.
\end{remark}

The main result of \cite{Shokurov} is the following theorem.

\begin{theorem}[{\cite[Main Theorem]{Shokurov}}]
\label{theorem:Shokurov}
In the notation and assumptions of Remark~\ref{remark:Prym}, suppose that $\Gamma$ is connected,
and the following condition holds:
\begin{itemize}
\item[(S)] for any splitting $\Gamma=\Gamma_1\cup\Gamma_2$,
one has $|\Gamma_1\cap\Gamma_2|\geqslant 4$.
\end{itemize}
Then $\mathrm{Prym}(\Gamma^\prime,\iota)$ is a sum of Jacobians of smooth curves if and only if $\Gamma$ is
\begin{itemize}
\item either hyperelliptic, or
\item trigonal, or
\item quasitrigonal, or
\item a plane quintic curve such that
$h^0(\Gamma',\mathcal L)$ is odd, where $\mathcal L$ ia a pull-back of $\left.\mathcal{O}_{\mathbb{P}^2}(1)\right|_{\Gamma}$
under the double cover $\Gamma'\to \Gamma$.
\end{itemize}
\end{theorem}

Thus, Theorems~\ref{theorem:Jac-Prym} and \ref{theorem:Shokurov} imply
the following result.

\begin{corollary}
\label{corollary:Shokurov}
Suppose that the curve $\Delta$ is connected and not hyperelliptic,
the curve~\mbox{$\kappa_\Delta(\Delta)$} is an intersection
of quadrics in $\mathbb{P}^N$,
and condition~$(\mathrm{S})$
of Theorem~\ref{theorem:Shokurov} holds for~$\Delta$.
Then $\mathrm{Prym}\left(\Delta^\prime,I\right)$ is not a sum of Jacobians of smooth curves.
\end{corollary}

\begin{proof}
By Remark~\ref{remark:hypertrig} and Theorem~\ref{theorem:Shokurov}, it is enough to show that $\Delta$ is not a plane quintic.
The latter follows from the fact that
quadrics in~$\P^5$ that pass through the canonical ima\-ge~$C$
of a plane nodal quintic
cut out a Veronese surface, so that $C$ is not an intersection of quadrics.
\end{proof}

Clemens and Griffiths proved in \cite[Corollary~3.26]{ClemensGriffiths} that $V$ is irrational provided that $\mathrm{J}(V)$ is not a sum of Jacobians of smooth curves.
Thus, Corollary~\ref{corollary:Shokurov} and Theorem~\ref{theorem:Jac-Prym} imply the following result.

\begin{corollary}
\label{corollary:Shokurov-2}
Suppose that the curve $\Delta$ is connected and not hyperelliptic,
the curve~\mbox{$\kappa_\Delta(\Delta)$} is an intersection
of quadrics in $\mathbb{P}^N$,
and condition~$(\mathrm{S})$
of Theorem~\ref{theorem:Shokurov} holds for~$\Delta$.
Then $V$ is irrational.
\end{corollary}

Corollary~\ref{corollary:Shokurov-2} and
Theorem~\ref{theorem:Shokurov} imply Conjecture~\ref{conjecture:Shokurov} for $U=\mathbb{F}_n$.
For details, see the proof of \cite[Theorem 10.2]{Shokurov}.

\begin{remark}
\label{remark:Shokurov}
In the notation and assumptions of Remark~\ref{remark:Prym}, suppose that there is a splitting
$$
\Gamma=E_1\cup\ldots\cup E_r\cup\Phi
$$
such that each $E_i$ is a smooth rational curve, the curves $E_1,\ldots,E_r$ are disjoint,
and each intersection $E_i\cap\Phi$ consists of two points.
Let $\Theta$ be a nodal curve obtained from $\Phi$ by gluing each pair of points $E_i\cap\Phi$.
It follows from \cite[Corollary~3.16]{Shokurov} and \cite[Remark~3.17]{Shokurov}
that there exists a connected nodal curve $\Theta^\prime$ together with an involution $\sigma$ on it
such that
$$\Theta^\prime/\sigma\cong\Theta,$$
the nodes of $\Theta^\prime$ are exactly the fixed points of $\sigma$,
the involution $\sigma$ does not interchange branches at these points, and
$$
\mathrm{Prym}\left(\Gamma^\prime,\iota\right)\cong\mathrm{Prym}\left(\Theta^\prime,\sigma\right).
$$
\end{remark}

\section{Quartic double solids and conic bundles}
\label{section:double-solids}

Let $\tau\colon X\to\mathbb{P}^3$ be a double
cover branched over a nodal quartic surface in $S$.
Suppose that $S$ is indeed singular,
and let $O_S$ be a singular point of the surface $S$.
Denote by $O_X$ the point in $X$ that is mapped to the point ${O}_S$ by the double cover $\tau$.
Then there exists a commutative diagram
$$
\xymatrix{
&&X\ar@{->}[rr]^{\tau}\ar@{-->}[ddrr]^{p_{O_X}}&&\mathbb{P}^3\ar@{-->}[dd]^{p_{{O}_S}}&&S\ar@{_{(}->}[ll]\ar@{-->}[ddll]^{p_{{O}_S}\vert_{S}}\\
{X_0}\ar@{->}[rru]^{f_{O_X}}\ar@{->}[drrrr]_{\pi}&&&&&&\\
&&&&\mathbb{P}^{2}&&}
$$
where $p_{O_S}$ is the linear projection from the point ${O}_S$,
the morphism $f_{O_X}$ is the blow up of the point $O_X$, the map
$p_{O_X}$ is undefined only in the point $O_X$, and $\pi$ is a conic bundle.
One has
$$
\left(\tau\circ f_{O_X}\right)^*\mathcal{O}_{\mathbb{P}^3}(1)-E_{O_X}\sim\pi^*\mathcal{O}_{\mathbb{P}^2}(1),
$$
where $E_{O_X}\cong\mathbb{P}^1\times\mathbb{P}^1$ is the exceptional surface of $f_{O_X}$.

\begin{remark}
\label{remark:X-blow-up}
The divisor $-K_{X_0}$ is ample and $-K_{X_0}^3=14$.
If $O_X$ is the only singular point of the surface $S$,
then $X_0$ is the Fano threefold No.~8
in the notation of~\cite[\S12.3]{IsPr99}.
\end{remark}

The restricted map
$$p_{O_S}\vert_{S}\colon S\dasharrow \mathbb{P}^2$$
is a generically two-to-one cover,
and its branch locus is a curve of degree~$6$,
which is also the degeneration curve of the conic bundle~$\pi$.
Denote this curve by~$C$.
The scheme fibers of $\pi$ over the points of the curve
$C$ are singular conics in~$\mathbb{P}^2$.
Note that the scheme fiber of $\pi$
over a point $\xi\in\mathbb{P}^2$ is non-reduced if and only if
the line $L_{\xi}$ mapped to $\xi$ by $p_{O_X}$ is contained in the quartic
$S$; in this case one obviously has~\mbox{$\xi\in C$}.

\begin{proposition}
\label{proposition:C6}
The singularities of the curve $C$ (if any) are nodes, cusps, or tacnodes.
Moreover, let $\xi$ be a point in $C$, and let $L_{\xi}$ be a line in $\mathbb{P}^3$ that is mapped to $\xi$ by the linear projection $p_{O_X}$.
Then the following assertions hold.
\begin{itemize}
\item[(i)] The curve $C$ has a tacnode at $\xi$ if and only if $L_{\xi}\subset S$, and
there are exactly two singular points of $S$ different from $O_S$ that are contained in $L_{\xi}$.

\item[(ii)] The curve $C$ has a cusp at $\xi$ if and only if $L_{\xi}\subset S$, and
there is a unique singular point of $S$ different from $O_S$ that is contained in $L_{\xi}$.

\item[(iii)] The curve $C$ has a node at $\xi$ if and only if one of the following two cases holds:
\begin{itemize}
\item[$\bullet$] $L_{\xi}\not\subset S$, and there is a unique singular point of $S$ different from $O_S$ that is contained in $L_{\xi}$;
\item[$\bullet$] $L_{\xi}\subset S$, and $O_S$ is the unique singular point of $S$ that is contained in $L_{\xi}$.
\end{itemize}
\end{itemize}
\end{proposition}

\begin{proof}
Choose homogeneous coordinates $x$, $y$, $z$, and $t$ in $\mathbb{P}^3$ so that
$O_S=[0:0:1:0]$. Then the quartic $S$ is given by equation
\begin{equation}\label{eq:quartic}
z^2q_2(x,y,t)+zq_3(x,y,t)+q_4(x,y,t)=0,
\end{equation}
where $q_i$ is a form of degree $i$.
One has
$$p_{O_S}([x:y:z:t])=[x:y:t].$$
Moreover, after a change of coordinates
$x$, $y$ and $t$ we may assume that the line $L_{\xi}$ is given by
equations $x=y=0$. The equation of the curve $C$ in a local
chart $\mathbb{A}^2\subset\mathbb{P}^2$
with coordinates $x$ and $y$ is written as $F(x,y)=0$, where
\begin{equation}\label{eq:degeneracy}
F(x,y)=q_3(x,y,1)^2-4q_2(x,y,1)q_4(x,y,1).
\end{equation}
Our further strategy is to consider several cases depending on the position of the line
$L_{\xi}$ with respect to the surface $S$, the number of singular
points of $S$ on $L_{\xi}$, and vanishing of certain coefficients in the equation of $S$,
and in each case to figure out a description of a singularity of the curve
$C$ at $\xi$.

\emph{Case I}.
Assume that the line $L_{\xi}$ is not contained in $S$, i.\,e. at least
one of the forms $q_i$ in~\eqref{eq:quartic}
contains a monomial~$t^i$ with non-zero coefficient.
We are going to show that $C$ is either smooth at $\xi$, or the situation
is described by the first option of case~(iii) in the assertion of the proposition.
Since $\xi$ is contained in $C$, we see that
$L_{\xi}$ intersects $S$ at $O_S$ and at most one more point.

\emph{Subcase I.1}.
Suppose that $O_S$ is the only common point of $L_{\xi}$ and $S$.
We are going to show that $C$ is smooth at $\xi$. One has
$q_4(0,0,1)\neq 0$ by~\eqref{eq:quartic}. Keeping this in mind
and looking at~\eqref{eq:quartic} once again, we see that
$$
q_2(0,0,1)=q_3(0,0,1)=0.
$$
Since $O_S$ is a node, we also see that
at least one of the partial derivatives
of $q_2$ with respect to $x$ and $y$ does not vanish at the
point $[0:0:1]$. We may assume that this is a partial derivative
with respect to $x$. Then \eqref{eq:degeneracy} implies that the partial derivative of $F$ with
respect to $x$ at the point $(0,0)$ does not vanish either,
so that $C$ is smooth at the point~$\xi$ in Subcase~I.1.

\emph{Subcase I.2}.
Suppose that the intersection $L_{\xi}\cap S$ contains a point $P_S$
different from~$O_S$.
Making a change of coordinates if necessary,
we may assume that
$$P_S=[0:0:0:1].$$
Then $q_4(0,0,1)=0$ by~\eqref{eq:quartic} and thus
$q_3(0,0,1)=0$ by~\eqref{eq:degeneracy}.
This implies~\mbox{$q_2(0,0,1)\neq 0$},
so that we may assume that $q_2(0,0,1)=1$.

\emph{Sub-subcase I.2(a)}.
Suppose that $P_S$ is a non-singular point of $S$.
We are going to show that $C$ is smooth at $\xi$.
At least one of the partial derivatives
of $q_4$ with respect to $x$ and $y$ does not vanish at the
point $[0:0:1]$. As above, we may assume that this is a partial derivative
with respect to $x$. Then \eqref{eq:degeneracy} implies that the partial derivative of $F$ with
respect to $x$ at the point $(0,0)$ does not vanish either,
so that $C$ is smooth at~$\xi$ in Sub-subcase I.2(a).

\emph{Sub-subcase I.2(b)}.
Suppose that $P_S$ is a singular point of $S$.
We are going to show that the situation
is described by the first option of case~(iii) in the assertion of the proposition.
None of the monomials of $q_4$ is divisible by $t^3$.
Write
$$q_3(x,y,t)=2l(x,y)t^2+\bar{q}_3(x,y,t)$$
and
$$q_4(x,y,t)=q(x,y)t^2+\bar{q}_4(x,y,t),$$
where $l$ is a linear form in $x$ and $y$,
$q$ is a quadratic form in $x$ and $y$, every monomial of $\bar{q}_3$
has degree at least $2$ in $x$, $y$,
while every monomial of $\bar{q}_4$
has degree at least~$3$ in $x$, $y$.
Regarding $x$, $y$, and $z$ as affine coordinates
in a neighborhood of the point~\mbox{$P_S=[0:0:0:1]$}, we can write
a local equation of (an affine chart on) $S$ in these coordinates
as
$$
g(x,y,z)+\Phi_{\geqslant 3}(x,y,z)=0,
$$
where
$$g(x,y,z)=z^2+2l(x,y)z+q(x,y),$$
and every monomial of $\Phi_{\geqslant 3}$ has degree at least~$3$.
Using the fact that $P_S$ is a node of~$S$, we conclude that
the quadratic form $g(x,y,z)$
is non-degenerate.
This means that~\mbox{$q(x,y)-l(x,y)^2$} is not a square.
On the other hand, we rewrite \eqref{eq:degeneracy} as
$$
F(x,y)=4l(x,y)^2-4q(x,y)+F_{\geqslant 3}(x,y),
$$
where any monomial of $F_{\geqslant 3}$ has degree at least $3$.
This implies that the curve $C$ has a node at~$\xi$ by definition of a nodal point,
so that in Sub-subcase~I.2(b)
we have the first option of case~(iii) of the assertion of the proposition.
Up to now we have completed Subcase~I.2, and the whole
Case~I.

\emph{Case II}.
Now assume that the line $L_{\xi}$ is contained in the quartic $S$.
We are going to show that the situation
is described either by case~(i), or by case~(ii), or by the second option of case~(iii) of the assertion of the proposition.
Our assumption implies that neither of the forms
$q_i$ in \eqref{eq:quartic} contains a monomial $t^i$ with non-zero coefficient.

\emph{Subcase II.1}.
Suppose that $L_{\xi}$ does not contain singular points of $S$ that
are different from $O_S$.
We are going to show that the situation
is described by the second option of case~(iii) in the assertion of the proposition.
Write
$$
q_i(x,y,t)=l_i(x,y)t^{i-1}+r_i(x,y,t)
$$
for $i=2,3,4$,
where $l_i$ are linear forms in $x$ and $y$, while
every monomial of $r_i$ has degree at least $2$ in $x,y$.
Put
$$
f(x,y)=l_3(x,y)^2-4l_2(x,y)l_4(x,y).
$$

We claim that $f(x,y)$ is not a square.
Indeed, suppose that $f(x,y)$ is a square of some linear form.
Note that $l_2(x,y)$ is not a zero polynomial since
$O_S$ is a node of $S$.
Since all partial derivatives of $r_i(x,y,t)$ vanish on the line $L_{\xi}$,
we see that $S$ has some singular point different from $O_S$
on $L_{\xi}$ provided that there is such point for
a cubic surface given by equation
$$
z^2l_2(x,y)+ztl_3(x,y)+t^2l_4(x,y)=0.
$$
The latter equation in $z$ and $t$ has two roots $\upsilon_1$ and $\upsilon_2$ over the field $\mathbb{C}(x,y)$
because its discriminant is a square, and each $\upsilon_i$ can be written as a quotient
of two linear forms in $x$ and $y$. However, both the sum and the product of $\upsilon_i$
are also quotients of linear functions in $x$ and $y$, which implies that
at least one of $\upsilon_i$, say $\upsilon_1$, is an element of~$\mathbb{C}$.
Thus the above cubic surface is singular at
the point $R=(0:0:\upsilon_1:1)$; actually, in this case the cubic surface is reducible.
Hence $S$ is also singular at the point $R$, which contradicts our current assumptions.

We see that $f(x,y)$ is not a square. Thus it is
a non-degenerate quadratic form in $x$ and~$y$. On the other hand, we
can rewrite~\eqref{eq:degeneracy} as
$$
F(x,y)=f(x,y)+F_{\geqslant 3}(x,y),
$$
where any monomial of $F_{\geqslant 3}$ has degree at least~$3$.
Therefore,
the curve $C$ has a node at~$\xi$ by definition of a nodal point, so that
in Subcase~II.1 we have the second option of case~(iii) of the assertion of the proposition.

\emph{Subcase II.2}.
Suppose that $L_{\xi}$ contains a singular point $P_S$ of $S$ such that
$P_S$ is different from $O_S$.
We are going to show that the situation
is described eithert by case~(i), or by case~(ii) in the assertion of the proposition.
Making a linear change of coordinates $z$ and~$t$, we may assume that
$P_S=[0:0:0:1]$ and we still have $O_S=[0:0:1:0]$. Since
$x$, $y$, and $t$ can be regarded as local affine coordinates in a neighborhood of $O_S$, and
the quadratic form $q_2$ is a non-degenerate form in $x$, $y$, and $t$,
we can make a further linear change of
variables $x$, $y$, and $t$, and assume that
\begin{equation}\label{eq:q2}
q_2=xt+y^2.
\end{equation}
Note that the coordinates of $O_S$ and $P_S$ remain unchanged under such coordinate
change.
Since $P_S$ is a singular point of $S$, the form $q_4$ does not
contain any of the monomials~$xt^3$ or~$yt^3$ with non-zero coefficient.
Recall also that the form $q_3$ does not
contain the monomial~$t^3$ with non-zero coefficient.
This implies that the form
$q_3$ contains at least one of the monomials~$xt^2$ or~$yt^2$
with non-zero coefficient. Indeed,
otherwise we can write
a local equation of (an affine chart on) $S$
as
$$
q_4(x,y,1)+\Phi_{\geqslant 3}(x,y,z)=0
$$
regarding $x$, $y$, and $z$ as affine coordinates
in a neighborhood of the singular point~\mbox{$P_S=[0:0:0:1]$},
where every monomial of $\Phi_{\geqslant 3}$ has degree at least $3$.
This means that the quadratic part of the latter equation
depends only on $x$ and $y$. This is impossible since
$P_S$ is a node of~$S$.

\emph{Sub-subcase II.2(a)}.
Suppose that $q_3$ contains the monomial $yt^2$ with
non-zero coefficient.
We are going to show that the curve $C$ has a cusp at $\xi$, so that
we are in case~(ii) of the assertion of the proposition.
It is easy to see from equation~\eqref{eq:quartic}
that~$O_S$ and~$P_S$ are the only singular points of $S$
contained in~$L_{\xi}$.

We can write
$$
q_3(x,y,t)=\tilde{l}(x,y)t^2+\tilde{q}_3(x,y,t),
$$
where $\tilde{l}$ is a linear form in $x$ and $y$ not proportional to $x$,
while every monomial of $\tilde{q}_3$ has degree at least $2$ in $x,y$.
For some linear form $l$ one has
$$
y=l(x,\tilde{l}(x,y)).
$$
Replacing the coordinate $y$ by $\tilde{l}(x,y)$
(and keeping the notation $y$ for the latter coordinate for simplicity),
we rewrite
$$
q_2(x,y,t)=xt+l(x,y)^2,\quad
q_3(x,y,t)=yt^2+\bar{q}_3(x,y,t),
$$
and
$$
q_4(x,y,t)=\alpha x^2t^2+\beta xyt^2+\gamma y^2t^2+\bar{q}_4(x,y,t),
$$
where $l$ is a linear form in $x$ and $y$,
every monomial of $\bar{q}_3$ has degree at least $2$ in $x,y$,
and every monomial of $\bar{q}_4$ has degree at least $3$ in $x,y$.
Regarding $x$, $y$, and $z$ as affine coordinates
in a neighborhood of the point $P_S=[0:0:0:1]$, we can write
a local equation of (an affine chart on) $S$ in these coordinates
as
$$
\alpha x^2+\beta xy+\gamma y^2+yz+\Phi_{\geqslant 3}(x,y,z)=0,
$$
where every monomial of $\Phi_{\geqslant 3}$ has degree at least $3$.
Since $P_S$ is a node of $S$, we have~\mbox{$\alpha\neq 0$}.
Assigning the weights $\wt(y)=3$ and $\wt(x)=2$,
we rewrite \eqref{eq:degeneracy} as
$$
F(x,y)=y^2-4\alpha x^3+F_{\geqslant 7}(x,y),
$$
where any monomial of $F_{\geqslant 7}$ has weight at least $7$.
Thus, the curve $C$ has a cusp
at the point~$\xi$ in Sub-subcase II.2(a);
this follows either from an explicit analytic change of coordinates, or from
a sufficient condition for a curve to have a cusp, see \cite[Theorem~II.13.2]{AVGZ}.

\emph{Sub-subcase II.2(b)}.
Finally, suppose that $q_3$ does not contain the monomial $yt^2$ with
non-zero coefficient.
We are going to show that the curve $C$ has a tacnode at $\xi$, so that
we are in case~(i) of the assertion of the proposition.
It is easy to see from equation~\eqref{eq:quartic}
that there is a unique singular point $Q_S$ of~$S$
different from $O_S$ and $P_S$ that is contained in $L_{\xi}$.
Since
$q_3$ does not contain the monomial $yt^2$ with
non-zero coefficient, we conclude that it contains
the monomial $xt^2$ with
some non-zero coefficient $\alpha$.
All other monomials in $q_3(x,y,t)$ have degree at least $2$ in $x$ and $y$.
Let $\epsilon$ be the coefficient at $y^2t$ in $q_3(x,y,t)$;
note that we do not claim that $\epsilon\neq 0$ here.
Thus we may write
\begin{equation}\label{eq:q3}
q_3(x,y,t)=\alpha xt^2+\epsilon y^2t+\bar{q}_3(x,y,t),
\end{equation}
where every monomial of $\bar{q}_3$
has degree at least $2$ in $x$, $y$, and is different from~$y^2t$.
Taking partial derivatives, it is easy to see
from equations~\eqref{eq:quartic} and~\eqref{eq:q3}
that the unique singular point of~$S$
different from $O_S$ and~$P_S$ that is contained in $L_{\xi}$
is
$$Q_S=[0:0:-\alpha:1].$$
Write
\begin{equation}\label{eq:q4}
q_4(x,y,t)=\beta y^2t^2+\gamma xyt^2+\delta x^2t^2+\bar{q}_4(x,y,t),
\end{equation}
where every monomial of $\bar{q}_4$
has degree at least $3$ in $x$, $y$.
Regarding $x$, $y$, and $z$ as affine coordinates
in a neighborhood of the point $P_S=[0:0:0:1]$, we can write
a local equation of (an affine chart on) $S$ in these coordinates
as
$$
\alpha xz+\beta y^2+\gamma xy+\delta x^2+\Phi_{\geqslant 3}(x,y,z)=0,
$$
where every monomial of $\Phi_{\geqslant 3}$ has degree at least $3$.
Using once again the fact that $P_S$ is a node of $S$,
we see that $\beta\neq 0$.
Choosing a new coordinate $z'=z+\alpha t$,
we rewrite~\eqref{eq:quartic} as
$$
\left(\alpha^2-\alpha\epsilon+\beta\right)y^2t^2+\Upsilon(x,y,z',t)=0,
$$
where every monomial of $\Upsilon$ either is divisible by $x$ or has degree at least $3$ in $x$, $y$ and $z'$.
Hence the fact that $Q_S$ is a node of $S$ implies that
$$
\alpha^2-\alpha\epsilon+\beta\neq 0.
$$
On the other hand, assigning the weights $\wt(x)=2$ and $\wt(y)=1$,
and using equations~\eqref{eq:q2}, \eqref{eq:q3} and~\eqref{eq:q4},
we rewrite~\eqref{eq:degeneracy} as
$$F(x,y)=F_4(x,y)+F_{\geqslant 5}(x,y),$$
where
$$
F_4(x,y)=\alpha^2 x^2+(2\alpha\epsilon-4\beta) xy^2+(\epsilon^2-4\beta) y^4,
$$
and every monomial of $F_{\geqslant 5}$ has weight at least $5$.
It is straightforward to check that the
polynomial~$F_4$ is not a square.
This means that the curve $C$ has a tacnode
at~$\xi$ in Sub-subcase II.2(b);
this follows either from an explicit analytic change of coordinates, or from
a sufficient condition for a curve to have a tacnode, see \cite[Theorem~II.13.2]{AVGZ}.
Thus we have completed Subcase~II.2, Case~II, and the proof of the proposition.
\end{proof}

\begin{remark}
\label{remark:Sarkisov-model}
If $O_S$ is the only singular point of the surface $S$, then Proposition~\ref{proposition:C6} follows
from a much more general~\cite[Corollary~1.11]{Sarkisov2}.
\end{remark}

\begin{lemma}\label{lemma:puchok}
In the notation of Proposition~\ref{proposition:C6}, suppose that the curve $C$ has a tacnode at the point $\xi$.
Let $F_\xi$ be the preimage of the point $\xi$ with respect to $\pi$.
Let $T$ be the line in $\P^2$ that passes through $\xi$ and has a local intersection number $4$ with the curve $C$ at $\xi$.
Denote by $\mathcal{B}$ the linear system  of conics in $\P^2$ that are tangent to $T$ at $\xi$.
Let $B_1$ and $B_2$ be preimages on $X_0$ of two general conics in $\mathcal{B}$.
Then each $B_i$ has a singularity locally
isomorphic to a product of a node and $\mathbb{A}^1$ at a general point of $F_\xi$,
and
$$
\mathrm{mult}_{F_\xi}(B_1\cdot B_2)=4.
$$
\end{lemma}

\begin{proof}
Using coordinates in $\P^3$ and $\P^2$ introduced in the proof of Proposition~\ref{proposition:C6},
we find that the line $T$ is given by equation $x=0$.
Regarding $x$ and $y$ as local coordinates in an affine chart containing $\xi$,
and making an analytic change of coordinates if  necessary,
we write an equation of a general conic in $\mathcal{B}$ as
$$
x-\lambda y^2=0,
$$
where $\lambda\in\mathbb{C}$.

Keeping in mind equation~\eqref{eq:quartic},
we write down the local equation of $X$ (and also of $X_0$ at a general point of $F_\xi$) in $\mathbb{A}^4$ as
$$
w^2=q_2(x,y,t)+q_3(x,y,t)+q_4(x,y,t).
$$
Using equations~\eqref{eq:q2}, \eqref{eq:q3} and~\eqref{eq:q4},
we see that the surfaces $B_i$ are locally defined by equations
$$
w^2=\mu_i y^2+F_i(y,t)
$$
in local coordinates $w$, $y$ and $t$,
where $\mu_i$ are (different)
non-zero constants, and every monomial of $F_i$ has degree at least $3$.
Since $F_\xi$ is given by $w=y=0$ in the same coordinates, the assertion of the lemma follows.
\end{proof}

\section{From non-standard to standard conic bundles}
\label{section:standard}

Let us use all notation and assumptions of
Section~\ref{section:double-solids}.
If $S$ is smooth away of $O_S$, then the conic bundle $\pi\colon X_0\to\mathbb{P}^2$ is standard by Theorem~\ref{theorem:Cheltsov}; indeed, in this case $X_0$ is smooth,
and Theorem~\ref{theorem:Cheltsov} tells us that $X$ is $\mathbb{Q}$-factorial, so that
the relative Picard group of $X_0$ over $\mathbb{P}^2$ has rank~$1$.
If there are other singular points of $S$ except $O_S$, then $X_0$ is singular
and thus the conic bundle $\pi\colon X_0\to\mathbb{P}^2$ is definitely not standard.
However,
it follows from~\mbox{\cite[Theorem~1.13]{Sarkisov2}}
that there exists a commutative diagram
\begin{equation}
\label{equation:standard-conic-bundle}
\xymatrix{
V\ar@{->}[d]_{\nu}\ar@{-->}[rr]^{\rho}&&{X}_0\ar@{->}[d]^{\pi}\\
U\ar@{->}[rr]^{\varrho}&&\mathbb{P}^2,}
\end{equation}
where $V$ is a smooth projective threefold, $U$ is a smooth surface, $\nu$ is a standard conic bundle,
$\rho$ is a birational map, and $\varrho$ is a birational morphism.
Of course, \eqref{equation:standard-conic-bundle} is not unique.
The goal of this section is to explicitly
construct~\eqref{equation:standard-conic-bundle} with
$\varrho$ being a composition of $|\mathrm{Sing}(S)|-1$ blow ups of smooth points.
Namely, we prove the following theorem.

\begin{theorem}
\label{theorem:conic-bundle-standard-weak-del-Pezzo}
Suppose that $X$ is $\mathbb{Q}$-factorial.
Then there exists a commutative
diagram~\eqref{equation:standard-conic-bundle} where $\nu$ is a standard
conic bundle and the following properties hold.
\begin{itemize}
\item[(i)] The birational morphism $\varrho$ is a composition of $|\mathrm{Sing}(S)|-1$ blow ups of smooth points.

\item[(ii)] The birational morphism $\varrho$ factors as
$$
U\stackrel{\varrho_t'}\longrightarrow
U_t\stackrel{\varrho_t}\longrightarrow
U_c\stackrel{\varrho_c}\longrightarrow
U_n\stackrel{\varrho_n}\longrightarrow\P^2,
$$
where the morphism $\varrho_n$ is a blow up of the nodes of the curve $C$ that are images of the singular points of $X_0$ via $\pi$,
the morphism $\varrho_c$ is a blow up of all cusps of the proper transform of $C$
on the surface $U_n$,
the morphism  $\varrho_t$ is a blow up of all tacnodes of the proper transform
of $C$ on the surface $U_c$,
and the morphism $\varrho_t^\prime$ is a blow up of all nodes of the proper
transform of $C$ on the surface $U_t$ that
are mapped to the tacnodes of the curve $C$ by
$\varrho_n\circ \varrho_c\circ \varrho_t$.
In particular, the birational map $\varrho^{-1}$ is regular
away of $\mathrm{Sing}(C)$.

\item[(iii)] Let $\Delta$ be the degeneration curve of the conic
bundle $\nu$. Then $\Delta$ is the proper transform of the curve $C$, i.~e.
the exceptional curves of $\varrho$ are not contained in $\Delta$.
In particular, one has $\Delta\sim -2K_U$.
\end{itemize}
\end{theorem}

In the rest of the section, we will prove
Theorem~\ref{theorem:conic-bundle-standard-weak-del-Pezzo}.
Namely, we will show how to construct the commutative
diagram \eqref{equation:standard-conic-bundle} by analyzing
the geometry of $X_0$ in a neighborhood of a fiber containing
a singular point of $X_0$, producing
a desired transformation in such neighborhood,
and then applying these constructions together to obtain a global
picture.

Let $\xi$ be a point of $C$, and let $F_\xi$ be the preimage of the point $\xi$ via $\pi$.
Let $L_{\xi}$ be a line in $\mathbb{P}^3$ that is mapped
to $\xi$ by the linear projection $p_{O_X}$,
so that $F_\xi$ is the preimage of $L_{\xi}$
via $\tau\circ f_{O_X}$.

Choose  homogeneous coordinates $x$, $y$, $z$ and $t$ in $\mathbb{P}^3$ so that
$O_S=[0:0:1:0]$ and the line $L_{\xi}$ is given by
equations $x=y=0$. One has
$$p_{O_S}([x:y:z:t])=[x:y:t].$$
During our next steps we will always assume that
the quartic $S$ is singular at some point $P_S$ of the line $L_\xi$
such that $P_S$ is different from $O_S$; we can choose $x$, $y$, $z$ and $t$
so that
$$P_S=[0:0:0:1].$$
Since $P_S$ is a node of $S$,
we know that $S$ is given by equation
\begin{equation}\label{eq:quartic-for-flops}
t^2q_2(x,y,z)+tq_3(x,y,z)+q_4(x,y,z)=0,
\end{equation}
where $q_i$ is a form of degree $i$ in three variables, and the quadratic form $q_2$ is non-degenerate.
We can expand~\eqref{eq:quartic-for-flops} as
\begin{multline}\label{eq:quartic-for-flops-details}
t^2\big(\alpha z^2+zq_2^{(1)}(x,y)+q_2^{(2)}(x,y)\big)+\\
+t\big(z^2q_3^{(1)}(x,y)+zq_3^{(2)}(x,y)+q_3^{(3)}(x,y)\big)+\\
+\big(z^2q_4^{(2)}(x,y)+zq_4^{(3)}(x,y)+q_4^{(4)}(x,y)\big)=0,
\end{multline}
where $q_i^{(j)}$ is a form of degree $j$ in two variables,
and $\alpha$ is a constant.

In the sequel we will frequently use the following easy and
well known auxiliary result.

\begin{lemma}\label{lemma:stupid}
Let $Y$ be a normal threefold, $R$ be a surface in $Y$, and $L$ be
a smooth rational curve in $R$ such that $R$ and $Y$ are smooth along $L$.
Suppose that $\mathcal{N}_{L/R}\cong\mathcal{O}_{\P^1}(r)$
and $-K_Y\cdot L=s$ with $2r\geqslant s-2$. Then
$$
\mathcal{N}_{L/Y}\cong\mathcal{O}_{\P^1}(r)\oplus\mathcal{O}_{\P^1}(s-r-2).
$$
\end{lemma}

\begin{proof}
We have
$$
\deg\mathcal{N}_{L/Y}=2g(L)-2-K_{Y}\cdot L=s-2.
$$
Also, there is an injective morphism $\mathcal{N}_{L/R} \hookrightarrow \mathcal{N}_{L/Y}$.
Therefore, there is an exact sequence of sheaves on $L\cong\mathbb{P}^1$
$$
0\to \mathcal{O}_{\mathbb{P}^1}(r)\to \mathcal{N}_{L/Y}\to
\mathcal{O}_{\mathbb{P}^1}(s-r-2)\to 0.
$$
Since $r\geqslant s-r-2$, the latter exact sequence splits and
gives the assertion of the lemma.
\end{proof}

Now we are ready to describe birational maps that are needed to transform
$\pi$ to a standard conic bundle.

\subsection*{Construction I}

Suppose that $S$ is singular at exactly two points $O_S$ and $P_S$ of the line~$L_\xi$,
and $L_\xi$ is not contained in $S$, so that the situation is described by the first option of case~(iii) of Proposition~\ref{proposition:C6}.
This happens if and only if one has $\alpha\neq 0$ in equation~\eqref{eq:quartic-for-flops-details}.
In particular, we can assume that
\begin{equation}\label{eq:quadric-I}
q_2(x,y,z)=xy+z^2.
\end{equation}

Denote by $P_0$ the preimage of the point $P_S$ on $X_0$.
The threefold $X_0$ has a node at $P_0$
and is smooth elsewhere along $F_\xi$.
The fiber $F_\xi$ consists of two smooth rational curves
that intersect transversally at the point $P_0$.

Let $f_{P_0}\colon X_1\to X_0$ be
the blow up of the point $P_0$,
and let $E_1$ be the exceptional divisor of~$f_{P_0}$.
One has $E_1\cong\mathbb{P}^1\times\mathbb{P}^1$,
and the threefold $X_1$ is smooth
along the proper transform of $F_\xi$.

Denote by $L_0^+$ and $L_0^-$ the irreducible components of $F_\xi$,
and denote by
$L_1^+$ and $L_1^-$ their proper transforms on $X_1$.
Then the curves ${L}_1^+$ and ${L}_1^-$ are disjoint smooth rational
curves.

\begin{lemma}\label{lemma:normal-bundle-1-1}
Let $L_1$ be one of the curves ${L}_1^+$ and ${L}_1^-$.
Then
$\mathcal{N}_{{L}_1/X_1}\cong\mathcal{O}_{\mathbb{P}^1}(-1)\oplus\mathcal{O}_{\mathbb{P}^1}(-1)$.
\end{lemma}

\begin{proof}
Let $\Pi\subset\mathbb{P}^3$ be a general plane containing the line $L_{\xi}$.
Then $\Pi$ is given by
$$
\lambda x+\mu y=0
$$
for some $[\lambda:\mu]\in\mathbb{P}^1$.
Let $R$ be the preimage of $\Pi$ via $\tau$.
We see from equation~\eqref{eq:quartic-for-flops-details}
that~$R$ has nodes at the preimages of the points $P_S$ and $O_S$ and is smooth elsewhere.

Let $R_1$ be the proper transform of $R$ on the threefold $X_1$.
Then the surface $R_1$ is smooth.
One has
$K_{R_1}\cdot L_1=-1$, so that $L_1^2=-1$ on $R_1$, and the normal bundle
$$
\mathcal{N}_{L_1/R_1}\cong\mathcal{O}_{\mathbb{P}^1}(-1).
$$
On the other hand, we know that
$$
K_{X_1}\cdot L_1=0,
$$
which implies the assertion by Lemma~\ref{lemma:stupid}.
\end{proof}

By Lemma~\ref{lemma:normal-bundle-1-1} one can flop each of the curves $L_1^+$ and $L_1^-$.
Namely, each of these two flops is just an Atiyah flop, i.\,e.
it can be obtained
by blowing up the curve $L_1^+$ or $L_1^-$ and blowing down the exceptional
divisor isomorphic to $\mathbb{P}^1\times\mathbb{P}^1$ along another ruling
onto a curve contained in a smooth locus of the resulting threefold (see \cite[\S4.2]{Kulikov}, \mbox{\cite[\S2]{Francia}}).
Let $\chi\colon X_1\dasharrow  X_2$
be the composition of Atiyah flops in the curves $L_1^+$ and $L_1^-$.
Let~$f_{\xi}\colon U_2\to \mathbb{P}^2$ be the blow up of the point $\xi$,
and $p_2\colon X_2\dasharrow U_2$ be the corresponding
rational map. Put $p_1=p_2\circ\chi$.

Let $Z\cong\mathbb{P}^1$ be the exceptional divisor of the blow up $f_{\xi}$, and
$C_2$ be the proper transform of the curve $C$ on $U_2$.
By Proposition~\ref{proposition:C6}(iii) the intersection $C_2\cap Z$
consists of two points, and $C_2$ is smooth at these points.
Let $E_2$ be the proper transform of the divisor
$E_1$ on the threefold~$X_2$.

\begin{lemma}\label{lemma:regular-1}
The rational map $p_2$ is a morphism, and $p_2(E_2)=Z$.
The fiber of $p_2$ over each of the two points
in $C_2\cap Z$ is a union of two smooth rational
curves that intersect transversally at one point.
All other fibers of $p_2$ over $Z$ are smooth, so that $C_2$
is the degeneration curve of the conic bundle $p_2$.
\end{lemma}

\begin{proof}
Denote by $\omega_1\colon W\to X_1$ the blow up of $X_1$ along the curves
$L_1^+$ and $L_1^-$, so that there is a commutative
diagram
$$
\xymatrix{
&W\ar@{->}[rd]^{\omega_1}\ar@{->}[ld]_{\omega_2}&\\
X_2&& X_1\ar@{-->}[ll]_{\chi}
}
$$
Let $G_W^+$ and $G_W^-$ be the exceptional
divisors of $\omega_1$ over the curves $L_1^+$ and $L_1^-$,
respectively. Recall that
$$G_W^+\cong G_W^-\cong\P^1\times\P^1.$$
Denote by $l_1^+$ and $l_1^-$ the classes of the rulings
of $G_W^+$ and $G_W^-$ that are mapped surjectively onto
$L_1^+$ and $L_1^-$ by $\omega_1$ (and are contracted by~$\omega_2$),
and denote by
$l_2^+$ and $l_2^-$ the classes of the rulings
of $G_W^+$ and $G_W^-$ that are contracted by~$\omega_1$.
Let $E_W^P$ be the proper transforms of the surface $E_1$ on~$W$.
One has
$$
E_W^P\vert_{G_W^+}\sim l_2^+,\quad
E_W^P\vert_{G_W^-}\sim l_2^-.
$$

Let $\mathcal{H}$ be the pencil of curves that are proper transforms
on $U_2$ of lines in $\P^2$ passing through the point $\xi$.
Note that the class of $H+R$, where $H\in\mathcal{H}$
and $R\in |f_{\xi}^*\mathcal{O}_{\P^2}(1)|$,
is very ample.
Note also that the proper transform on $X_2$ of the linear
system $|f_{\xi}^*\mathcal{O}_{\P^2}(1)|$ is base point free.
Thus, to conclude that
the rational map $p_2$ is a morphism it is enough to check that the proper transform
$\mathcal{H}_{X_2}$ of the linear system $\mathcal{H}$ on $X_2$
has no base points.

Let us first show that the proper transform $\mathcal{H}_W$ of the pencil $\mathcal{H}_{X_2}$ on $W$ is base point free.
By construction, its base locus is contained in the union~\mbox{$G_W^+\cup G_W^-\cup E_W^P$}.
One has
$$
\mathcal{H}_W\sim \left(\pi\circ f_{P_0}\circ\omega_1\right)^*\mathcal{O}_{\mathbb{P}^2}(1)-G_W^{+}-G_W^{-}-E_W^P.
$$
This gives $\mathcal{H}_W\vert_{G_W^+}\sim l_1^+$ and $\mathcal{H}_W\vert_{G_W^-}\sim l_1^-$.
Therefore, either two different elements of the pencil
$\mathcal{H}_W$ do not have intersection points in $G_W^+$, or all of them contain one and the same ruling of class $l_1^+$.
The latter is impossible since the proper transforms of elements
of $\mathcal{H}$ on $X_1$ are transversal to each other
at a general point of $L_1^+$.
Thus, $\mathcal{H}_W$ does not have base points on $G_W^+$.
In a similar way we see that it does not have base points
on $G_W^-$.

Let us check that the pencil $\mathcal{H}_W$ has no base points in $E_W^P$.
It is most convenient to do this by analyzing the behavior of the rational map $p_1$ along the surface $E_1$.
Using equation~\eqref{eq:quadric-I} and writing down the equation of $X$, we see that the surface $E_1$ is identified with a quadric surface
given by
$$
xy+z^2=w^2
$$
in $\mathbb{P}^3$ with homogeneous coordinates $x$, $y$, $z$, and $w$.
Note that $x$ and $y$ can be interpreted as homogeneous coordinates on $Z$.
The closure of the image of $E_1$ with respect to the rational map $p_1$
is the curve $Z$. The restriction $p_{E_1}$ of $p_1$ to $E_1$
is given by
\begin{equation}
\label{equation:linear-projection-1}
[x:y:z:w]\mapsto [x:y].
\end{equation}
Therefore, $p_{E_1}$ is a projection from the line $x=y=0$, which intersects $E_1$ at the points~\mbox{$[0:0:1:1]$} and~\mbox{$[0:0:1:-1]$}.
Note that these are the points $P_1^+={L}_1^+\cap E_1$ and~$P_1^-=L_1^-\cap E_1$, up to relabelling
$P_1^+$ and $P_1^-$.
This implies that the pencil $\mathcal{H}_W$ has no base points in $E_W^P$ except possibly
in the two curves contracted to $P_1^+$ and $P_1^-$ by $\omega_1$.
But these curves are contained in the
divisors $G_W^+$ and $G_W^-$, respectively, and we already know
that  $\mathcal{H}_W$ has no base points in these surfaces.
Thus, the pencil $\mathcal{H}_W$ is base point free.
In particular, we see that $p_2\circ\omega_2$ is a morphism.

The restrictions of $\mathcal{H}_W$ to the surfaces $G_W^+$ and $G_W^-$ are contained in the fibers
of the contraction $\omega_2$. This shows that the pencil $\mathcal{H}_{X_2}$ is also base point free,
so that $p_2$ is a morphism.

The remaining assertions of the lemma
follow from~\eqref{equation:linear-projection-1}.
\end{proof}

Putting everything together, we obtain a commutative diagram
$$
\xymatrix{
&W\ar@{->}[rd]^{\omega_1}\ar@{->}[ld]_{\omega_2}&\\
X_2\ar@{->}[dd]_{p_2}&&X_1\ar@{-->}[ddll]^{p_1}\ar@{->}[drr]^{f_{P_0}}\ar@{-->}[ll]_{\chi}\\
&&&&X_0\ar@{->}[d]^{\pi}\\
{U_2}\ar@{->}[rrrr]^{f_\xi}&&&&\mathbb{P}^2,}
$$

\subsection*{Construction II}

Suppose that $S$ is singular at exactly two points $O_S$ and $P_S$ of the line $L_\xi$,
and $L_\xi$ is contained in $S$, so that the situation is described by case~(ii) of Proposition~\ref{proposition:C6}.\
This happens if and only if in equation~\eqref{eq:quartic-for-flops-details}
one has~\mbox{$\alpha=0$}, and the linear forms
$q_2^{(1)}(x,y)$ and $q_3^{(1)}(x,y)$ are not proportional.
In particular, we can assume that~$q_2^{(1)}(x,y)=x$ and
\begin{equation}\label{eq:quadric-II}
q_2(x,y,z)=xz+y^2.
\end{equation}

As in Construction~I, denote by $P_0$
the preimage on $X_0$ of the point $P_S$. Note that
$F_\xi$ is a smooth rational curve passing through $P_0$.
The threefold $X_0$ has a node at $P_0$
and is smooth elsewhere along $F_\xi$.

Let $f_{P_0}\colon X_1\to X_0$ be
the blow up of the point $P_0$,
and let $E_1$ be the exceptional divisor of $f_{P_0}$.
One has $E_1\cong\mathbb{P}^1\times\mathbb{P}^1$.
Denote by $L_1$ the proper transform of $F_\xi$ on $X_1$.
The threefold $X_1$ is smooth along $L_1$.

We need the following auxiliary result which is actually easy and well known.

\begin{lemma}\label{lemma:normal-bundle-aux}
Let $Y$ be a normal threefold, and $C$ be a smooth rational curve contained in the smooth
locus of $Y$. Let $P$ be a point on $C$, and $h\colon Y'\to Y$ be the blow up of $P$.
Let $C'$ be the proper transform of $C$ on $Y'$. Write
$\mathcal{N}_{C'/Y'}\cong\mathcal{O}_{\mathbb{P}^1}(a)\oplus\mathcal{O}_{\mathbb{P}^1}(b)$.
Then
$$
\mathcal{N}_{C/Y}\cong\mathcal{O}_{\mathbb{P}^1}(a+1)\oplus\mathcal{O}_{\mathbb{P}^1}(b+1).
$$
\end{lemma}
\begin{proof}
Suppose that $\mathcal{N}_{C/Y}\cong\mathcal{O}_{\mathbb{P}^1}(c)\oplus\mathcal{O}_{\mathbb{P}^1}(d)$.
One has
\begin{equation}\label{eq:diff}
(d+c)-(a+b)=\deg\mathcal{N}_{C/Y}-\deg\mathcal{N}_{C/Y}=2.
\end{equation}

Let $g\colon W\to Y$ be the blow up of the curve $C$, and $G$ be the exceptional divisor of~$g$. Then $G$ is a Hirzebruch surface $\mathbb{F}_r$,
where $r=|c-d|$. Let $Z$ be the fiber of the pro\-jec\-ti\-on~$g\vert_G\colon G\to C$
over the point $P$.
Let $g'\colon W'\to Y'$ be the blow up of the curve $C'$,
and $G'$ be the exceptional divisor of $g'$. Then $G'$ is a Hirzebruch surface $\mathbb{F}_{r'}$,
where~$r'=|a-b|$.

Note that there is commutative diagram:
$$
\xymatrix{
&W'\ar@{->}[ld]_{g'}\ar@{->}[rd]^{h'}&\\
Y'\ar@{->}[rd]_h&&W\ar@{->}[ld]^g\\
&Y&
}
$$
where $h'\colon W'\to W$ that is a blow up of the curve $Z$.
Its existence follows from purely local computations near the point $P$.
In particular, the surface $G'$ is the proper
transform of the surface $G$ with respect to $h'$, so that
$G\cong G'$. Thus we have
$$
|c-d|=r=r'=|a-b|,
$$
and applying~\eqref{eq:diff} we obtain the assertion of the lemma.
\end{proof}

As in the proof of Lemma~\ref{lemma:normal-bundle-1-1}, let $\Pi\subset\mathbb{P}^3$ be a general
plane containing the line $L_{\xi}$, and let $R$ be the preimage of $\Pi$ via $\tau$.
Denote by $R_0$ the proper transform of $R$ on the threefold $X_0$.
Then it follows from equation~\eqref{eq:quartic-for-flops-details}
that $R_0$ has a node at the point $P_0$,
one more node at some point $P_{\Pi}\in F_\xi$, and is smooth elsewhere.
One has
$$
K_{R_0}\cdot F_\xi=K_{X_0}\cdot F_\xi=-1.
$$

\begin{lemma}\label{lemma:normal-bundle-0-2}
One has
$\mathcal{N}_{L_1/X_1}\cong\mathcal{O}_{\mathbb{P}^1}\oplus\mathcal{O}_{\mathbb{P}^1}(-2)$.
\end{lemma}

\begin{proof}
Let $f\colon X_1'\to X_1$ be the blow up of the preimage on $X_1$ of the point $P_{\Pi}$.
Put~\mbox{$f'=f_{P_0}\circ f$},
so that $f'\colon X_1'\to X_0$ is the blow up of the points $P_0$ and $P_{\Pi}$.
Let $R_1'$ be the proper transform of $R_0$ on the threefold $X_1'$.
Then the surface $R_1'$ is smooth.
Note that the morphism
$$f'\vert_{R_1'}\colon R_1'\to R_0$$
is the blow up of nodes
of $R_0$,
and thus it is crepant.
Let $L_1'$ be the proper transform of $F_\xi$ (or $L_1$) on $X_1'$.
Let $E_1'$ be the exceptional divisor of $f'$ over the point $P_0$ (i.~e. the proper transform
on $X_1'$ of the exceptional divisor of $f_{P_0}$), and
$E'$ be the exceptional divisor of $f'$ over the point $P_{\Pi}$.

One has $K_{R_1'}\cdot L_1'=-1$, so that $L_1'^2=-1$, and the normal bundle
$\mathcal{N}_{L_1'/R_1'}\cong\mathcal{O}_{\mathbb{P}^1}(-1)$.
On the other hand, we know that
$$
K_{X_1'}\cdot L_1'=\big(f'^*K_{X_0}+E_1'+2E'\big)\cdot L_1'=2,
$$
which gives
$$
\mathcal{N}_{L_1'/X_1'}\cong \mathcal{O}_{\mathbb{P}^1}(-1)\oplus\mathcal{O}_{\mathbb{P}^1}(-3).
$$
by Lemma~\ref{lemma:stupid}.
Now the assertion follows from Lemma~\ref{lemma:normal-bundle-aux}.
\end{proof}

Let $f_1\colon\bar{X}_1\to X_1$ be the blow up of the curve $L_1$, and let $\bar{G}_1$ be its exceptional surface.
By Lemma~\ref{lemma:normal-bundle-0-2} we have $\bar{G}_1\cong\mathbb{F}_2$.
Denote by $\bar{L}_1$ the unique smooth rational
curve in $\bar{G}_1$ such that~\mbox{$\bar{L}_1^2=-2$}.

\begin{lemma}\label{lemma:0-2-to-1-1}
One has
$\mathcal{N}_{\bar{L}_1/\bar{X}_1}\cong\mathcal{O}_{\mathbb{P}^1}(-1)\oplus\mathcal{O}_{\mathbb{P}^1}(-1)$.
\end{lemma}
\begin{proof}
Let $\bar{R}_1$ and $\bar{E}_1$ be the proper transforms on $\bar{X}_1$ of the surfaces $R$ and $E_1$, respectively.
Let
$\bar{E}_0$ be the proper transform of the exceptional divisor of
the blow up~\mbox{$f_{O_X}\colon X_0\to X$} on $\bar{X}_1$.
Then
$$
\bar{R}_1\sim\big(f_{O_X}\circ f_0\circ f_1\big)^*(R)-\bar{E}_0-\bar{E}_1-\bar{G}_1,
$$
because $R$ has nodes at the points $O_X$ and $f_{O_X}(P_0)$,
and it is smooth at the general point of the curve $f_{O_X}(F_\xi)$.
Denote by $l$ the class of the fiber of the natural
projection~\mbox{$\bar{G}_1\to L_1\cong\P^1$}
in $\mathrm{Pic}(\bar{G}_1)$. Then
$$
\bar{E}_0\vert_{\bar{G}_1}\sim \bar{E}_1\vert_{\bar{G}_1}\sim l.
$$
Moreover, we have $\bar{G}_1\vert_{\bar{G}_1}\sim-\bar{L}_1-2l$.
Therefore, one has
$$
\bar{R}_1\vert_{\bar{G}_1}=\Big(\big(f_{O_X}\circ f_0\circ f_1\big)^*(R)-\bar{E}_0-\bar{E}_1-\bar{G}_1\Big)\Big\vert_{\bar{G}_1}
\sim \bar{L}_1+\big(f_{O_X}\circ f_0\circ f_1\big)^*(R)\vert_{\bar{G}_1}\sim \bar{L}_1+l.
$$

On the other hand, we know that the proper transform of
$R$ on $X_1$ is a del Pezzo surface with a unique node on
the curve $L_1$. Thus, we conclude that
$\bar{R}_1\vert_{\bar{G}_1}=L+T$, where $L$ is a fiber of the
projection $\bar{G}_1\to L_1$, and $T$ is some effective one-cycle
such that $T\sim_{\Q} \bar{L}_1$.
Since $\bar{L}_1$ is an irreducible curve and $\bar{L}_1^2=-2<0$,
this immediately implies that $T=\bar{L}_1$, because $T$ is an effective one-cycle.

Since $f_0\circ f_1\vert_{\bar{R}_1}\colon \bar{R}_1\to R_0$
is the minimal resolution of singularities of a nodal del Pezzo surface
$R_0$, we have $K_{\bar{R}_1}\cdot\bar{L}_1=-1$.
Therefore, one has
$$
\mathcal{N}_{\bar{L}_1/\bar{R}_1}\cong\mathcal{O}_{\P^1}(-1).
$$
Finally, we have $K_{\bar{X}_1}\cdot\bar{L}_1=0$,
so that the assertion follows by Lemma~\ref{lemma:stupid}.
\end{proof}

By Lemma~\ref{lemma:0-2-to-1-1}
one can make an Atiyah flop $\psi\colon\bar{X}_1\dasharrow\bar{X}_2$
in the curve $\bar{L}_1$.
Let $\bar{G}_2$ be the proper transform of
the surface $\bar{G}_1$ on the threefold $\bar{X}_2$.
\begin{lemma}
\label{lemma:F-2}
One has $\bar{G}_2\cong\mathbb{F}_2$.
\end{lemma}

\begin{proof}
Denote by $\omega_1\colon W\to \bar{X}_1$ the blow up of $\bar{X}_1$
along the curve $\bar{L}_1$, so that there is a commutative
diagram
$$
\xymatrix{
&W\ar@{->}[rd]^{\omega_1}\ar@{->}[ld]_{\omega_2}&\\
\bar{X}_2&& \bar{X}_1\ar@{-->}[ll]_{\psi}
}
$$
Let $G_W$ be the exceptional divisor of $\omega_1$. Recall that $G_W\cong\P^1\times\P^1$.
Denote by $l_1$ the class of the ruling of $G_W$ that is mapped surjectively onto $\bar{L}_1$,
and denote by $l_2$ the class of the ruling of $G_W$ that is contracted by $\omega_1$ to a point in $\bar{L}_1$.
Let $G_{1,W}$ be the proper transforms on $W$ of the surface $\bar{G}_1$, respectively.
Then $G_{1,W}\cong\bar{G}_1\cong\mathbb{F}_2$ and
$$
G_{1,W}\sim\omega_1^*(\bar{G}_1)-G_W.
$$
Since $G_W\vert_{G_W}\sim -l_1-l_2$ and $\bar{G}_1\cdot\bar{L}_1=0$, this gives
$$
G_{1,W}\vert_{G_W}\sim\Big(\omega_1^*\bar{G}_1-G_W\Big)\vert_{G_W}\sim l_1+l_2.
$$
Thus, the morphism $\omega_2$ induces an isomorphism $G_{1,W}\cong\bar{G}_2$, so that $\bar{G}_2\cong\mathbb{F}_2$.
\end{proof}

We have the following commutative diagram:
\begin{equation}
\label{equation:small-pagoda}
\xymatrix{
\bar{X}_2\ar@{->}[d]_{f_2}&&\bar{X}_1\ar@{->}[d]^{f_1}\ar@{-->}[ll]_{\psi}\\
X_2&&X_1\ar@{-->}[ll]_{\chi}.}
\end{equation}
Note that $\chi$ is a flop, and $L_1$
is a $(-2)$-curve of width $2$ in the notation
of~\cite[Definition~5.3]{Reid1983}.
The diagram~\eqref{equation:small-pagoda} is an example
of a \emph{pagoda} described in~\cite[5.7]{Reid1983}.

Let $f_{\xi}\colon {U_2}\to\mathbb{P}^2$ be the blow up of the point $\xi$,
and $p_2\colon X_2\dasharrow {U_2}$ be the corresponding
rational map. Put $p_1=p_2\circ\chi$.
Let $Z\cong\mathbb{P}^1$ be the exceptional divisor of the blow up~$f_{\xi}$, and
$C_2$ be the proper transform of the curve $C$ on $U_2$.
By Proposition~\ref{proposition:C6}(ii) the intersection $C_2\cap Z$
consists of a single point, and $C_2$ is smooth at this point.
Let $E_2$ be the proper transform of the divisor
$E_1$ on the threefold $X_2$.

Now we will prove a result that is identical to Lemma~\ref{lemma:regular-1}
(but takes place in the setup of our current Construction~II).

\begin{lemma}\label{lemma:regular-2}
The rational map $p_2$ is a morphism, and $p_2(E_2)=Z$.
The fiber of $p_2$ over the point $C_2\cap Z$ is a union of two smooth rational
curves that intersect transversally at one point.
All other fibers of $p_2$ over $Z$ are smooth, so that $C_2$ is the degeneration curve of the conic bundle $p_2$.
\end{lemma}

\begin{proof}
Denote by $\omega_1\colon W\to \bar{X}_1$ the blow up of $\bar{X}_1$
along the curve $\bar{L}_1$, so that there is a commutative
diagram
$$
\xymatrix{
&W\ar@{->}[rd]^{\omega_1}\ar@{->}[ld]_{\omega_2}&\\
\bar{X}_2&& \bar{X}_1\ar@{-->}[ll]_{\psi}
}
$$
Let $G_W$ be the exceptional divisor of $\omega_1$. Recall that $G_W\cong\P^1\times\P^1$.
Denote by $l_1$ the class of the ruling of $G_W$ that is mapped surjectively onto $\bar{L}_1$.
Let $E_W^P$ and $G_{1,W}$ be the proper transforms on $W$
of the surfaces $E_1$ and $\bar{G}_1$, respectively.

As in the proof of Lemma~\ref{lemma:regular-1}, let
$\mathcal{H}$ be the pencil of the curves that are proper transforms
on $U_2$ of lines in $\P^2$ passing through the point $\xi$.
To show that
the rational map $p_2$ is a morphism,
it is enough to check that the proper transform
$\mathcal{H}_{X_2}$ of the linear system $\mathcal{H}$ on $X_2$
is base point free.
Let us show first that its proper transform $\mathcal{H}_W$ on the threefold $W$ is base point free.

By construction, we know that all base points of the pencil
$\mathcal{H}_W$ are contained in the
union~\mbox{$G_W\cup G_{1,W}\cup E_W^P$}.
One has
$$
\mathcal{H}_W\sim \left(\pi\circ f_{P_0}\circ f_1\circ\omega_1\right)^*\mathcal{O}_{\mathbb{P}^2}(1)-G_W -G_{1,W}-E_W^P.
$$
This gives
$\mathcal{H}_W\vert_{G_W}\sim l_1$.
Therefore, either two different elements of the pencil
$\mathcal{H}_W$ do not have intersection points in $G_W$,
or all of them contain one and the same ruling of class~$l_1$.
The latter case is impossible; indeed, the proper transforms of elements
of $\mathcal{H}$ on~$\bar{X}_1$ are tangent to each other along $\bar{L}_1$ with multiplicity~$2$
since $\tau$ is a double cover and the proper transforms
of the elements of $\mathcal{H}$ on $\P^3$ are planes
passing through the line $L_{\xi}$.
There\-fore,~$\mathcal{H}_W$ has no base points in $G_W$.

Let $t_1$ be the class of the ruling of $G_{1,W}\cong\mathbb{F}_2$.
Then
$$
\mathcal{H}_W\vert_{G_{1,W}}\sim t_1,
$$
and the rulings of $G_{1,W}$ cut out by the members of the pencil $\mathcal{H}_W$ vary (cf. the proof of Lemma~\ref{lemma:0-2-to-1-1}).
Therefore, $\mathcal{H}_W$ has no base points in $G_{1,W}$.

Let us check that $\mathcal{H}_W$ has no base points in $E_W^P$
by analyzing the behavior of the rational map $p_1$ along the surface $E_1$.
Using equation~\eqref{eq:quadric-II}  and writing down the equation
of $X$, we see that the surface $E_1$ is identified with a quadric surface
given by
$$
xz+y^2=w^2
$$
in $\mathbb{P}^3$ with homogeneous coordinates $x$, $y$, $z$, and $w$.
Note that $x$ and $y$ can be interpreted as homogeneous coordinates on $Z$.
The closure of the image of $E_1$ with respect to the rational map $p_1$
is the curve $Z$. The restriction $p_{E_1}$ of $p_1$ to $E_1$
is given by the formula~\eqref{equation:linear-projection-1}.
Therefore, $p_{E_1}$ is a projection from
the line $x=y=0$, which is tangent to $E_1$
at the point~\mbox{$[0:0:1:0]$}.
Note that this is the point
$P_1=L_1\cap E_1$.
This implies that $\mathcal{H}_W$ has no base points in $E_W^P$ outside
the curves contracted to $P_1$ by $f_1\circ\omega_1$.
But these are exactly the curves $G_{1,W}\cap E_W^P$ and $G_W\cap E_W^P$.
Since we already know that $\mathcal{H}_W$ has no base points in $G_{1,W}$ and $G_W$,
we conclude that $\mathcal{H}_W$ is base point free.
In particular, the rational map~\mbox{$p_2\circ f_2\circ\omega_2$}
is a morphism.

Since $\mathcal{H}_W\vert_{G_W}\sim l_1$, the proper transform $\mathcal{H}_{\bar{X}_2}$
of the pencil $\mathcal{H}_{X_2}$ on the threefold $\bar{X}_2$ is also base point free.
Let $t_2$ be the class of the ruling of $\bar{G}_2\cong\mathbb{F}_2$.
Then
$$
\mathcal{H}_{\bar{X}_2}\vert_{\bar{G}_2}\sim t_2,
$$
so that the restriction of $\mathcal{H}_{\bar{X}_2}$ to $\bar{G}_2$
lies in the fibers of the morphism~$f_2$.
Therefore, the pencil $\mathcal{H}_{X_2}$  is also base point free,
so that $p_2$ is a morphism.

The remaining assertions of the lemma
follow from~\eqref{equation:linear-projection-1}.
\end{proof}

Putting everything together, we obtain a commutative diagram
$$
\xymatrix{
&W\ar@{->}[rd]^{\omega_1}\ar@{->}[ld]_{\omega_2}&\\
\bar{X}_2\ar@{->}[d]_{f_2}&&\bar{X}_1\ar@{->}[d]^{f_1}\ar@{-->}[ll]_{\psi}\\
 X_2\ar@{->}[dd]_{p_2}&&X_1\ar@{-->}[ddll]^{p_1}\ar@{->}[drr]^{f_{P_0}}\ar@{-->}[ll]_{\chi}\\
&&&&X_0\ar@{->}[d]^{\pi}\\
U_2\ar@{->}[rrrr]^{f_\xi}&&&&\mathbb{P}^2}
$$

\subsection*{Construction III}

Suppose that $S$ is singular at exactly three points of the line $L_\xi$,
namely $O_S$, $P_S$, and some other point $Q_S$ different from $O_S$ and $P_S$;
in particular, this implies that $L_\xi$ is contained in $S$. Here
the situation is described by case~(i) of Proposition~\ref{proposition:C6}.
This happens if and only if in equation~\eqref{eq:quartic-for-flops-details}
one has~\mbox{$\alpha=0$},
and the linear forms~$q_2^{(1)}(x,y)$
and $q_3^{(1)}(x,y)$ are proportional.
In particular, we can assume that $q_2(x,y,z)$
is given by equation~\eqref{eq:quadric-II}.

Denote by $P_0$ and $Q_0$
the preimages of the points $P_S$ and $Q_S$ on $X_0$.
Note that $F_\xi$ is a smooth rational curve passing through $P_0$ and $Q_0$.
The threefold $X_0$ has nodes at $P_0$ and~$Q_0$,
and is smooth elsewhere along $F_\xi$.

Let $f\colon X_1\to X_0$ be
the blow up of the points $P_0$ and $Q_0$.
Denote by $E_1^P$ and $E_1^Q$ be the exceptional divisors of $f$
over the points $P_0$ and $Q_0$, respectively.
One has
$$
E_1^P\cong E_1^Q\cong\mathbb{P}^1\times\mathbb{P}^1.
$$
Denote by $L_1$ the proper transform of $F_\xi$ on $X_1$.
The threefold $X_1$ is smooth along $L_1$.

\begin{lemma}\label{lemma:normal-bundle-1-2}
One has
$\mathcal{N}_{L_1/X_1}\cong\mathcal{O}_{\mathbb{P}^1}(-1)\oplus\mathcal{O}_{\mathbb{P}^1}(-2)$.
\end{lemma}

\begin{proof}
As in the proof of Lemmas~\ref{lemma:normal-bundle-1-1} and~\ref{lemma:normal-bundle-0-2},
let $\Pi\subset\mathbb{P}^3$ be a general plane containing the line $L_{\xi}$,
let $R$ be the preimage of $\Pi$ with respect to the double cover $\tau$,
and let $R_0$ be the proper transform of $R$ on the threefold $X_0$.
Since the points $O_S$, $P_S$ and $Q_S$ are nodes of the surface $S$,
the intersection $\Pi\cap S$ consists of a line $L_{\xi}$ and a smooth cubic curve that intersects the line $L_{\xi}$ transversally at the points $O_S$, $P_S$ and $Q_S$.
Therefore $R_0$ has nodes at the points $P_0$ and $Q_0$, and is smooth elsewhere.
One has
$$
K_{R_0}\cdot F_\xi=K_{X_0}\cdot F_\xi=-1.
$$

Let $R_1$ be the proper transform of $R_0$ on the threefold $X_1$.
Then the surface $R_1$ is smooth.
The morphism $f\vert_{R_1}\colon R_1\to R_0$ is the blow up of nodes of $R_0$, and thus it is crepant.
One has $K_{R_1}\cdot L_1=-1$, so that $L_1^2=-1$ and
$$
\mathcal{N}_{L_1/R_1}\cong\mathcal{O}_{\mathbb{P}^1}(-1).
$$
On the other hand, we know that
$$
K_{X_1}\cdot L_1=\big(f^*K_{X_0}+E_1^P+E_1^Q\big)\cdot L_1=1,
$$
which implies the assertion by Lemma~\ref{lemma:stupid}.
\end{proof}

By Lemma~\ref{lemma:normal-bundle-1-2} and \cite[\S2]{Francia},
there exists an antiflip $\sigma\colon X_1\dasharrow\check{X}_1$ in the curve $L_1$.
The inverse map $\sigma^{-1}$ is usually called a Francia flip.

Let $f_{\xi}\colon U_1\to \mathbb{P}^2$ be the blow up of the point $\xi$.
Let $Z_1\cong\mathbb{P}^1$
be the exceptional divisor of the blow up $f_{\xi}$, and
$C_1$ be the proper transform of the curve $C$ on $U_1$.
By Proposition~\ref{proposition:C6}(i), the intersection $C_1\cap Z_1$
consists of a single point $\xi_1$, and $C_1$ has a node at $\xi_1$.
Let $p_1\colon X_1\dasharrow U_1$ and $\check{p}_1\colon\check{X}_1\dasharrow U_1$ be the resulting rational maps.
In fact, the rational map $\check{p}_1$ is a morphism.
To prove this, we need to recall the explicit construction of $\sigma$ from \cite[\S2]{Francia}.

Let $f_1\colon\bar{X}_1\to X_1$ be the blow up of the curve $L_1$, and let $\bar{G}_1$ be its exceptional surface.
By Lemma~\ref{lemma:normal-bundle-1-2} we have $\bar{G}_1\cong\mathbb{F}_1$. Denote by $\bar{L}_1$ the unique smooth rational
curve in $\bar{G}_1$ such that $\bar{L}_1^2=-1$.

\begin{lemma}\label{lemma:1-2-to-1-1}
One has $\mathcal{N}_{\bar{L}_1/\bar{X}_1}\cong\mathcal{O}_{\mathbb{P}^1}(-1)\oplus\mathcal{O}_{\mathbb{P}^1}(-1)$.
\end{lemma}

\begin{proof}
One has
$$\mathcal{N}_{\bar{L}_1/\bar{G}_1}\cong\mathcal{O}_{\mathbb{P}^1}(-1)$$
by construction.
On the other hand, we know that~\mbox{$K_{\bar{X}_1}\cdot \bar{L}_1=0$},
which implies the assertion by Lemma~\ref{lemma:stupid}.
\end{proof}

By Lemma~\ref{lemma:1-2-to-1-1}, we can make an Atiyah flop $\psi\colon\bar{X}_1\dasharrow\hat{X}_1$ in the curve $\bar{L}_1$.
Thus, there is a commutative diagram
$$
\xymatrix{
&W\ar@{->}[rd]^{\omega_1}\ar@{->}[ld]_{\omega_2}&\\
\hat{X}_1&& \bar{X}_1\ar@{-->}[ll]_{\psi},
}
$$
where $\omega_1$ is the blow up of the curve $\bar{L}_1$, and $\omega_2$ is the contraction
of the exceptional di\-vi\-sor~$G_W\cong\mathbb{P}^1\times\mathbb{P}^1$ of $\omega_1$ onto a smooth rational curve $\hat{L}_1$
contained in the smooth locus of~$\hat{X}_1$.
Denote by $E_W^P$, $E_W^Q$, and $G_{1,W}$ the proper transforms on $W$ of the sur\-fa\-ces~$E_1^P$,~$E_1^Q$, and~$\bar{G}_1$, respectively.

Let $\hat{G}_1$ be the proper transform of $\bar{G}_1$ on $\hat{X}_1$.
Then $\hat{G}_1\cong\mathbb{P}^2$ and its normal bundle in $\hat{X}_1$ is isomorphic to $\mathcal{O}_{\mathbb{P}^2}(-2)$,
so that there exists a contraction $g_1\colon \hat{X}_1\to\check{X}_1$ of the surface $\hat{G}_1$ to a singular point $\Xi_1$ of type $\frac{1}{2}(1,1,1)$.
There is a commutative diagram
\begin{equation}
\label{equation:half-diagram}
\xymatrix{
&W\ar@{->}[rd]^{\omega_1}\ar@{->}[ld]_{\omega_2}&\\
\hat{X}_1\ar@{->}[d]^{g_1}&&\bar{X}_1\ar@{->}[d]^{f_1}\ar@{-->}[ll]_{\psi}\\
\check{X}_1\ar@{-->}[dd]^{\check{p}_1}&&X_1\ar@{-->}[ll]_{\sigma}\ar@{-->}[ddll]^{p_1}\ar@{->}[dr]^{f}&\\
&&&X_0\ar@{->}[d]^{\pi}\\
U_1\ar@{->}[rrr]^{f_\xi}&&&\mathbb{P}^2.}
\end{equation}

\begin{lemma}\label{lemma:regular-2-1}
The rational map $\check{p}_1$ is a morphism.
\end{lemma}

\begin{proof}
As in the proof of Lemmas~\ref{lemma:regular-1} and \ref{lemma:regular-2}, let
$\mathcal{H}$ be the pencil of curves that are proper transforms on $U_1$ of lines in $\P^2$ passing through the point $\xi$.
Denote by $\mathcal{H}_{\check{X}_1}$ its proper transform on $\check{X}_1$.
To show that $\check{p}_1$ is a morphism,
it is enough to show that $\mathcal{H}_{\check{X}_1}$ is base point free.
To start with, we show that its proper transform $\mathcal{H}_W$ on $W$ is base point free.

By construction, we know that all base points of the pencil $\mathcal{H}_W$ are contained in the
union
$$
G_W\cup G_{1,W}\cup E_W^P\cup E_W^Q.
$$
Let us show that the pencil $\mathcal{H}_W$ has no base points in these surfaces.

Let $l_1$ be the class of the ruling of $G_W\cong\mathbb{P}^1\times\mathbb{P}^1$ that is contracted by $\omega_2$.
We already showed in the proof of Lemma~\ref{lemma:normal-bundle-1-2} that
the proper transforms of general surfaces of $\mathcal{H}$ on $X_0$ have nodes in $P_0$ and $Q_0$,
and the proper transforms of general surfaces of $\mathcal{H}$ on $X_1$ are smooth and contain the curve $L_1$.
Moreover, arguing as in the proof of Lemma~\ref{lemma:0-2-to-1-1}, we see
the proper transforms of general surfaces of $\mathcal{H}$ on $\bar{X}_1$ pass through the curve $\bar{L}_1$.
This implies that
$$
\mathcal{H}_W\sim \left(\pi\circ f\circ f_1\circ\omega_1\right)^*\mathcal{O}_{\mathbb{P}^2}(1)-2G_W-G_{1,W}-E_W^P-E_W^Q.
$$
This implies that $\mathcal{H}_W\vert_{G_W}\sim l_1$.
Therefore, either two different elements of the pen\-cil~$\mathcal{H}_W$ do not have intersection points in $G_W$,
or all of them contain one and the same ruling of class~$l_1$.
The latter case is impossible.
Indeed, the proper transforms of elements of $\mathcal{H}$ on~$X_1$ are tangent to each other along $L_1$ with multiplicity~$2$,
so that their proper transforms on $\bar{X}_1$ intersect each other transversally at general point of the curve $\bar{L}_1$,
which implies that two different elements of the pencil $\mathcal{H}_W$ cannot both contain a curve that is mapped dominantly to $\bar{L}_1$ by $\omega_1$.
Therefore, the pencil $\mathcal{H}_W$ has no base points in $G_W$.
Also, we have
$$
\mathcal{H}_W\vert_{G_{1,W}}\sim 0,
$$
which implies that the surface $G_{1,W}$ is disjoint from a general member of the pencil $\mathcal{H}_W$.
In particular, $\mathcal{H}_W$ has no base points in $G_{1,W}$.

Arguing as in the proof of Lemma~\ref{lemma:regular-2}, we see that
the pencil $\mathcal{H}_W$ does not have base points in the surfaces $E_W^P$ and $E_W^Q$
outside the curves
$$E_W^P\cap G_W,\quad E_W^P\cap G_{1,W},\quad E_W^Q\cap G_W,\quad
E_W^Q\cap G_{1,W}.$$
But we already know that $\mathcal{H}_W$ has no base points in $G_W$ and $G_{1,W}$.
This shows that $\mathcal{H}_W$ is base point free.
In particular, the rational map $\check{p}_1\circ g_1\circ\omega_2$ is a morphism.

Observe that the restrictions
$$
\mathcal{H}_W\vert_{G_W}\sim G_{1,W}\vert_{G_W}\sim l_1
$$
lie in the fibers of the morphism $\omega_2$.
Therefore, the proper transform $\mathcal{H}_{\hat{X}_1}$ of the pen\-cil~$\mathcal{H}_{\check{X}_1}$ on the threefold $\hat{X}_1$
is base point free, and the surface $\hat{G}_1$ is disjoint from its general member.
This shows that $\mathcal{H}_{\check{X}_1}$ is base point free, so that $\check{p}_1$ is a morphism.
\end{proof}

Let us describe the fibers of $\check{p}_1$ over the points of the curve $Z_1$.
Denote by $\bar{E}_1^P$, $\hat{E}_1^P$, and~$\check{E}_1^P$ the proper transforms of the surface $E_1^P$
on the threefolds $\bar{X}_1$, $\hat{X}_1$, and $\check{X}_1$,
respectively. Similarly, denote by
$\bar{E}_1^Q$, $\hat{E}_1^Q$, and
$\check{E}_1^Q$ the proper transforms of the sur\-fa\-ce~$E_1^Q$
on the threefolds $\bar{X}_1$, $\hat{X}_1$, and $\check{X}_1$,
respectively. One has
$$
\check{E}_1^P\cap\check{E}_1^Q=\hat{L}_1.
$$
Moreover, the surfaces $\check{E}_1^P$ and $\check{E}_1^Q$ intersect
along the curve~$\hat{L}_1$, and this intersection
is transversal outside the singular point~$\Xi_1$.
Furthermore, one has
$$
\check{p}_1^{-1}(Z_1)=\check{E}_1^P\cup\check{E}_1^Q.
$$

The curve $L_1$ intersects each of the divisors $E_1^P$ and $E_1^Q$ transversally at a single point.
Denote by $M_P$ and $M'_P$ the two rulings of $E_1^P\cong\P^1\times\P^1$
that pass through the intersection point $L_1\cap E_1^P$,
and denote by $M_Q$ and $M'_Q$ the two rulings of $E_1^Q\cong\P^1\times\P^1$
that pass through the intersection point $L_1\cap E_1^Q$.
Let $\check{M}_P$, $\check{M}'_P$, $\check{M}_Q$, and $\check{M}'_Q$
be the proper transforms on $\check{X}_1$ of the curves $M_P$, $M'_P$, $M_Q$, and $M'_Q$, respectively.
Then the cur\-ves~$\check{M}_P$,~$\check{M}'_P$,~$\check{M}_Q$,
and~$\check{M}'_Q$ pass through the singular point $\Xi_1$,
and are mapped by $\check{p}_1$ to the nodal point $\xi_1$ of the curve $C_1$.
Since
$$
K_{\hat{X}_1}\sim_{\mathbb{Q}} g_1^{*}K_{\check{X}_1}+\frac{1}{2}\hat{G}_1,
$$
one has
$$
-K_{\check{X}_1}\cdot\check{M}_P=-K_{\check{X}_1}\cdot\check{M}'_P=-K_{\check{X}_1}\cdot\check{M}_Q=-K_{\check{X}_1}\cdot\check{M}'_Q=\frac{1}{2}.
$$
This shows that $\check{M}_P+\check{M}'_P+\check{M}_Q+\check{M}'_Q$ is a scheme theoretic fiber of $\check{p}_1$ over $\xi_1$.
All other fibers of $\check{p}_1$ over the points of $Z_1$ are described by the following remark.

\begin{remark}
\label{remark:E-P-Q}
The commutative diagram \eqref{equation:half-diagram} gives the commutative diagram
$$
\xymatrix{
&E_W^P\ar@{->}[rd]^{\omega_1}\ar@{->}[ld]_{\omega_2}&\\
\hat{E}_1^P\ar@{->}[d]_{g_1}&&\bar{E}_1^P\ar@{->}[d]^{f_1}\\
\check{E}_1^P\ar@{->}[dr]_{\check{p}_1}&&E_1^P\ar@{-->}[dl]^{p_1}&\\
&Z_1.&&}
$$
Here we denote the restrictions of the morphisms $\omega_1$, $\omega_2$, $g_1$, $f_1$, $\check{p}_1$,
and the rational map~$p_1$ to the corresponding surfaces by the same symbols for simplicity.
The surface $E_1^P$ can be identified with a quadric in $\mathbb{P}^3$,
and the rational map $p_1$ is the linear projection of $E_1^P$ from a line
that is tangent to it at the point $P_1=L_1\cap E_1^P$ (cf. the proof of Lemma~\ref{lemma:regular-2-1}).
The morphism $f_1$ is the blow up of the point $P_1$,
the morphism $\omega_1$ is the blow up of the point $\bar{L}_1\cap\bar{E}_1^P$,
the morphism $\omega_2$ is an isomorphism.
The mor\-phism~$g_1$ is the contraction of the
$(-2)$-curve \mbox{$\hat{G}_1\cap\hat{E}_1^P$}
to the node $\Xi_1$ of the surface $\check{E}_1^P$.
By construction, we have~$\check{p}_1(\Xi_1)=\xi_1$, and the fiber of~$\check{p}_1$ over $\xi_1$ is $\check{M}_P\cup\check{M}'_P$.
The fibers of~$\check{p}_1$ over all other points in $Z_1$ are smooth rational curves.
A similar description applies to the
sur\-fa\-ces~$E_1^Q$,~$\bar{E}_1^Q$,~$E_W^Q$,~$\hat{E}_1^Q$, and~$\check{E}_1^Q$.
\end{remark}

Let $\hat{M}_P$, $\hat{M}'_P$, $\hat{M}_Q$, and $\hat{M}'_Q$
be the proper transforms on $\hat{X}_1$ of the curves
$M_P$, $M'_P$, $M_Q$, and $M'_Q$, respectively.
The curves $\hat{M}_P$, $\hat{M}'_P$, $\hat{M}_Q$, and $\hat{M}'_Q$
are pairwise disjoint, and each of them is disjoint from the curve $\hat{L}_1$.

\begin{lemma}\label{lemma:4-times-1-1}
Each of the normal bundles
$\mathcal{N}_{\hat{M}_P/\hat{X}_1}$,
$\mathcal{N}_{\hat{M}'_P/\hat{X}_1}$,
$\mathcal{N}_{\hat{M}_Q/\hat{X}_1}$, and
$\mathcal{N}_{\hat{M}'_Q/\hat{X}_1}$
is isomorphic to
$\mathcal{O}_{\mathbb{P}^1}(-1)\oplus\mathcal{O}_{\mathbb{P}^1}(-1)$.
\end{lemma}

\begin{proof}
Let $\bar{M}_P$, $\bar{M}'_P$, $\bar{M}_Q$, and $\bar{M}'_Q$
be the proper transforms on $\bar{X}_1$ of the cur\-ves~$M_P$,~$M'_P$,~$M_Q$, and $M'_Q$, respectively.
Then $\bar{M}_P$, $\bar{M}'_P$, $\bar{M}_Q$,
and $\bar{M}'_Q$ are disjoint from the curve~$\bar{L}_1$.
Therefore, to prove the assertion of the lemma, it is enough to compute the normal bundles of the latter four curves on $\bar{X}_1$.

One has
$\mathcal{N}_{\bar{M}_P/\bar{E}_1^P}\cong\mathcal{O}_{\mathbb{P}^1}(-1)$.
On the other hand, we compute
$K_{\bar{X}_1}\cdot \bar{M}_P=0$, so that
the assertion for the curve $\bar{M}_P$ follows from Lemma~\ref{lemma:stupid}.
For the curves $\bar{M}'_P$, $\bar{M}_Q$ and $\bar{M}'_Q$
the argument is similar.
\end{proof}

By Lemma~\ref{lemma:4-times-1-1},
we can make simultaneous Atiyah flops
in the curves $\hat{M}_P$, $\hat{M}'_P$, $\hat{M}_Q$, and $\hat{M}'_Q$.
Let $\phi\colon\hat{X}_1\dasharrow\bar{X}_2$ be the composition of these four flops.

Let $\hat{\omega}_1\colon \hat{W}\to\hat{X}_1$
be the blow up of the curves $\hat{M}_P$, $\hat{M}'_P$, $\hat{M}_Q$,
and~$\hat{M}'_Q$.
Denote by~$N_P$,~$N_P^\prime$,~$N_Q$, and $N_Q^\prime$ the exceptional surfaces of $\hat{\omega}_1$ that are mapped to
the cur\-ves~$\hat{M}_P$,~$\hat{M}'_P$,~$\hat{M}_Q$, and $\hat{M}'_Q$, respectively.
Then there is a commutative diagram
$$
\xymatrix{
&\hat{W}\ar@{->}[rd]^{\hat{\omega}_1}\ar@{->}[ld]_{\hat{\omega}_2}&\\
\bar{X}_2&&\hat{X}_1\ar@{-->}[ll]_{\phi},}
$$
where $\hat{\omega}_2$ is the contraction of the surfaces $N_P$, $N_P^\prime$, $N_Q$, and $N_Q^\prime$ to
smooth rational curves contained in the smooth locus of $\bar{X}_2$.

Let $f_{\xi_1}\colon U_2\to U_1$ be the blow up of the point $\xi_1$.
Denote by $T_2$ the exceptional divisor of the blow up $f_{\xi_1}$,
and by $Z_2$ and $C_2$ the proper transforms of the curves
$Z_1$ and $C$ on the surface $U_2$, respectively.
Note that the curves $C_2$ and $Z_2$ are disjoint.
Furthermore, let
$\bar{p}_2\colon\bar{X}_2\dasharrow U_2$ be the resulting rational map.
We have constructed the  following commutative diagram
\begin{equation}
\label{equation:second-half-diagram}
\xymatrix{
&&&\hat{W}\ar@{->}[ld]_{\hat{\omega}_2}\ar@{->}[rd]^{\hat{\omega}_1}&&&\\
&&\bar{X}_2\ar@{-->}[dd]^{\bar{p}_2}&&\hat{X}_1\ar@{->}[drr]^{g_1}\ar@{-->}[ll]_{\phi}&&&\\
&&&&&&\check{X}_1\ar@{->}[d]^{\check{p}_1}&&&\\
 &&U_2\ar@{->}[rrrr]^{f_{\xi_1}}&&&&U_1.&&&}
\end{equation}

\begin{lemma}\label{lemma:regular-3-1}
The rational map $\bar{p}_2$ is a morphism.
\end{lemma}

\begin{proof}
Let $\mathcal{B}_{U_1}$ be the linear subsystem in $|f_\xi^*\mathcal{O}_{\mathbb{P}^2}(2)-Z_1|$
consisting of all curves that pass through the point $\xi_1$.
Note that the base locus of $\mathcal{B}_{U_1}$ is the point $\xi_1$.
Moreover, the point $\xi_1$ is a scheme theoretic
intersection of curves in $\mathcal{B}_{U_1}$.
Denote by $\mathcal{B}_{U_2}$ the proper transform of
$\mathcal{B}_{U_1}$ on $U_2$, so that $\mathcal{B}_{U_2}$ is a base point
free linear system.

Denote by $\mathcal{B}_{\check{X}_1}$ the proper transform
of $\mathcal{B}_{U_2}$ on $\check{X}_1$ via $\check{p}_1$.
Then the base locus of $\mathcal{B}_{\check{X}_1}$ consists of the curves  $\check{M}_P$, $\check{M}'_P$, $\check{M}_Q$, and $\check{M}'_Q$.
Moreover, the union of these curves is a scheme theoretic intersection of surfaces in $\mathcal{B}_{\check{X}_1}$.

Denote by $\mathcal{B}_{\bar{X}_2}$ the proper transform of  $\mathcal{B}_{\check{X}_1}$ on the threefold $\bar{X}_2$.
To prove that $\bar{p}_2$ is a morphism, it is enough to show that $\mathcal{B}_{\bar{X}_2}$ is base point free.
Denote by $\mathcal{B}_{\bar{X}_1}$, $\mathcal{B}_{\hat{X}_1}$, and~$\mathcal{B}_{\hat{W}}$ the proper transforms of  $\mathcal{B}_{\check{X}_1}$ on the threefolds $\bar{X}_1$, $\hat{X}_1$, and $\hat{W}$, respectively.
To show that $\mathcal{B}_{\bar{X}_2}$ is base point free, let us describe the base loci of $\mathcal{B}_{\bar{X}_1}$, $\mathcal{B}_{\hat{X}_1}$, and $\mathcal{B}_{\hat{W}}$.

We claim that the base locus of $\mathcal{B}_{\bar{X}_1}$ consists of the curves $\bar{M}_P$, $\bar{M}'_P$, $\bar{M}_Q$, and $\bar{M}'_Q$.
We already know that these curves are contained in the base locus.
On the other hand, the base locus of $\mathcal{B}_{\hat{X}_1}$ is contained in the union of
the curves $\hat{M}_P$, $\hat{M}'_P$, $\hat{M}_Q$, and $\hat{M}'_Q$, and the surface $\hat{G}_1$.
Thus, the base locus of $\mathcal{B}_{\bar{X}_1}$ consists of the curves $\bar{M}_P$, $\bar{M}'_P$, $\bar{M}_Q$, and $\bar{M}'_Q$,
and a (possibly empty) subset of $\bar{G}_1$. Using Lemma~\ref{lemma:puchok}, we obtain the equivalence
$$
\mathcal{B}_{\bar{X}_1}\sim \left(\pi\circ f\circ f_1\right)^*\mathcal{O}_{\mathbb{P}^2}(2)-2\bar{G}_1-\bar{E}_1^P-\bar{E}_1^Q.
$$
This gives
$$
\mathcal{B}_{\bar{X}_1}\vert_{\bar{G}_1}\sim 2\bar{L}_1+2t,
$$
where $t$ is the class of a ruling of $\bar{G}_1\cong\mathbb{F}_1$.
The latter equivalence together with
Lemma~\ref{lemma:puchok} shows that the restriction $\mathcal{B}_{\bar{X}_1}\vert_{\bar{G}_1}$
does not have base curves that are mapped dominantly
to the curve $L_1$ by $f_1$.
In particular, a general surface in $\mathcal{B}_{\bar{X}_1}$ is disjoint from the curve~$\bar{L}_1$.
On the other hand, the four points
$$\bar{M}_P\cap \bar{G}_1,\quad \bar{M}'_P\cap \bar{G}_1,\quad
\bar{M}_Q\cap \bar{G}_1, \quad \bar{M}'_Q\cap \bar{G}_1$$
are contained in the base locus of the restriction $\mathcal{B}_{\bar{X}_1}\vert_{\bar{G}_1}$.
This implies that $\mathcal{B}_{\bar{X}_1}\vert_{\bar{G}_1}$ is a pencil whose base locus consists of exactly these four points.
In particular, the base locus of $\mathcal{B}_{\bar{X}_1}$ consists of the curves $\bar{M}_P$, $\bar{M}'_P$, $\bar{M}_Q$, and $\bar{M}'_Q$.

Since a general surface in $\mathcal{B}_{\bar{X}_1}$ is disjoint from $\bar{L}_1$, we see that
the base locus of $\mathcal{B}_{\hat{X}_1}$ consists of the curves $\hat{M}_P$, $\hat{M}'_P$, $\hat{M}_Q$, and $\hat{M}'_Q$.
By the same reason, we see that the restriction of $\mathcal{B}_{\hat{X}_1}$ to $\hat{G}_1\cong\mathbb{P}^2$
is a pencil of conics that pass through the four points
$$
\hat{M}_P\cap\hat{G}_1,\quad \hat{M}'_P\cap\hat{G}_1,\quad \hat{M}_Q\cap\hat{G}_1,\quad \hat{M}'_Q\cap\hat{G}_1.
$$
In particular, these four points
are in general position.

Computing the classes of the restrictions of the linear system $\mathcal{B}_{\hat{W}}$ to the exceptional divisors
$N_P$, $N_P^\prime$, $N_Q$, and $N_Q^\prime$, we see that they all lie in the fibers of $\hat{\omega}_2$.
This shows that both linear systems $\mathcal{B}_{\hat{W}}$ and $\mathcal{B}_{\bar{X}_2}$ are base point free.
Thus, we proved that  $\bar{p}_2$ is a morphism.
\end{proof}

Let $\bar{G}_2$, $\bar{E}_2^P$, and $\bar{E}_2^Q$ be the proper transforms
on $\bar{X}_2$ of the surfaces $\hat{G}_1$, $\hat{E}_1^P$,
and~$\hat{E}_1^Q$, respectively.
Then $\bar{E}_2^P$ (respectively, $\bar{E}_2^Q$)
is isomorphic to $\mathbb{F}_1$ since it is obtained from the
surface~$\hat{E}_1^P$ (respectively,
$\hat{E}_1^Q$) by blowing down two $(-1)$-curves $\hat{M}_P$ and $\hat{M}'_P$
(respectively,~$\hat{M}_Q$ and $\hat{M}'_Q$) as a result of flopping them (cf. Remark~\ref{remark:E-P-Q}).
Similarly, $\bar{G}_2$ is a smooth del Pezzo surface of degree $5$
since it is obtained from the surface $\hat{G}_1\cong\mathbb{P}^2$
by blowing up four points
$$
\hat{M}_P\cap\hat{G}_1,\quad \hat{M}'_P\cap\hat{G}_1,\quad \hat{M}_Q\cap\hat{G}_1,\quad \hat{M}'_Q\cap\hat{G}_1,
$$
which are in general position (see the proof of Lemma~\ref{lemma:regular-3-1}).

We know that the preimage of $Z_2$ via $\bar{p}_2$ is the union of the surfaces $\bar{E}_2^P$ and $\bar{E}_2^Q$.
These surfaces intersect transversally along the curve that is a unique $(-1)$-curve on each of them.
Moreover, the restrictions
$$\bar{p}_2\vert_{\bar{E}_2^P}\colon \bar{E}_2^P\to Z_2$$
and
$$\bar{p}_2\vert_{\bar{E}_2^Q}\colon \bar{E}_2^Q\to Z_2$$
are just natural projections of $\bar{E}_2^P\cong\mathbb{F}_1$
and $\bar{E}_2^Q\cong\mathbb{F}_1$ to $\mathbb{P}^1$.
In particular, the curve $Z_2$ is contained
in the degeneration curve of the conic bundle~$\bar{p}_2$.

By construction, the preimage of the curve $T_2$ via $\bar{p}_2$ is the surface $\bar{G}_2$,
and the induced morphism
$\bar{p}_2\vert_{\bar{G}_2}\colon \bar{G}_2\to T_2$
is a conic bundle with three reducible fibers.
In particular, the curve $T_2$ is not contained in the degeneration curve of the conic bundle  $\bar{p}_2$,
so that the latter degeneration curve is $Z_2\cup C_2$.
Note that the fibers of $\bar{p}_2$ over the two points in~$C_2\cap T_2$ must be reducible.
Also the fiber of $\bar{p}_2$ over the point $T_2\cap Z_2$ is reducible.
Thus, these three fibers are all reducible fibers of $\bar{p}_2$ over $T_2$.

There are contractions
$$f_2^P\colon\bar{X}_2\to X_2^P$$
and
$$f_2^Q\colon\bar{X}_2\to X_2^Q$$
of the surfaces
$\bar{E}_2^P$ and $\bar{E}_2^Q$ to the curves contained in the smooth loci of the threefolds $X_2^P$ and $X_2^Q$, respectively.
Let $p_2^P\colon X_2^P\to U_2$ and $p_2^Q\colon X_2^Q\to U_2$ be the resulting morphisms.
Then there is a commutative diagram
\begin{equation}
\label{equation:third-half-diagram}
\xymatrix{
&\bar{X}_2\ar@{->}[dd]^{\bar{p}_2}\ar@{->}[rd]^{f^P_2}\ar@{->}[ld]_{f^Q_2}&\\
X_2^Q\ar@{->}[dr]_{p_2^Q}&&X_2^P\ar@{->}[dl]^{p_2^P}\\
&U_2.&}
\end{equation}

\begin{corollary}\label{corollary:regular-3}
The curve $C_2$ is the degeneration curve of both conic bundles $p_2^P$ and~$p_2^Q$.
\end{corollary}

Gluing together the commutative diagrams \eqref{equation:half-diagram}, \eqref{equation:second-half-diagram}, and \eqref{equation:third-half-diagram}, we obtain a commutative diagram
$$
\xymatrix{
&&&\hat{W}\ar@{->}[ld]_{\hat{\omega}_2}\ar@{->}[rd]^{\hat{\omega}_1}&&W\ar@{->}[ld]_{\omega_2}\ar@{->}[rd]^{\omega_1}&\\
&&\bar{X}_2\ar@{->}[dddd]^{\bar{p}_2}\ar@{->}[rrdd]^{f^P_2}\ar@{->}[lldd]_{f^Q_2}&&\hat{X}_1\ar@{->}[drr]^{g_1}\ar@{-->}[ll]_{\phi}&&\bar{X}_1\ar@{->}[drr]^{f_1}\ar@{-->}[ll]_{\psi}&\\
&&&&&&\check{X}_1\ar@{->}[ddd]^{\check{p}_1}&&X_1\ar@{-->}[ll]_{\sigma}\ar@{-->}[dddll]^{p_1}\ar@{->}[ddr]^{f}&\\
 X_2^Q\ar@{->}[ddrr]_{p_2^Q}&&&& X_2^P\ar@{->}[ddll]^{p_2^P}&&&&&\\
&&&&&\smiley&&&&X_0\ar@{->}[d]^{\pi}\\
&&U_2\ar@{->}[rrrr]^{f_{\xi_1}}&&&&U_1\ar@{->}[rrr]^{f_\xi}&&&\mathbb{P}^2}
$$

\subsection*{}

Now we are ready to finish the proof.

\begin{proof}[{Proof of Theorem~\ref{theorem:conic-bundle-standard-weak-del-Pezzo}}]
Note that  Constructions~I, II, and III are local over the base of  the conic bundle.
Thus, they are applicable not only to the conic bundle $\pi$ but also to any other conic bundle
which is obtained from $\pi\colon X_0\to\P^2$ by a birational transformation that is local over the base,
provided that they are carried out over neighborhoods of points of the base not influenced by this transformation.

We start with the conic bundle $\pi\colon X_0\to\P^2$.
Keeping in mind Proposition~\ref{proposition:C6} and
applying Construction~I in neighborhoods of points of $\mathbb{P}^2$ where the curve
$C$ has a node and over which the fiber
of $\pi$ contains a singular point
of $X_0$, we obtain the birational
morphism~\mbox{$\varrho_n\colon U_n\to \mathbb{P}^2$}.
Applying Construction~II in neighborhoods of points of $U_n$
where the proper transform of $C$ has cusps,
we obtain the birational morphism~\mbox{$\varrho_c\colon U_c\to U_n$}.
Finally, applying Construction~III in neighborhoods of points of $U_c$
where the proper transform of $C$ has tacnodes, we obtain
the birational morphisms~\mbox{$\varrho_t\colon U_t\to U_c$}
and~\mbox{$\varrho_t'\colon U\to U_t$}.
We put
$$
\varrho=\varrho_n\circ\varrho_c\circ\varrho_t\circ\varrho_t'.
$$
We denote by $\rho\colon V\dasharrow X_0$ the birational map provided
by Constructions~I, II, and~III. Choosing the conic bundle $p_2\colon X_2\to U_2$
after performing Constructions~I or II, and choosing any of the conic bundles $p_2^P\colon X_2^P\to U_2$
or $p_2^Q\colon X_2^Q\to U_2$ after performing Construction~III,
we finally obtain a conic bundle~\mbox{$\nu\colon V\to U$}.
This completes the diagram~\eqref{equation:standard-conic-bundle}.

We already know that the threefold $V$ and the surface $U$ are smooth.
Also, we know from Lefschetz theorem that
$\mathrm{rk~Pic}(X)=1$, so that
$\mathrm{rk~Cl}(X)=1$ by $\mathbb{Q}$-factoriality assumption.
Keeping track of the blow ups we make in course of our
construction, we see that at the starting point we have
an equality $\mathrm{rk~Cl}(X_0/\mathbb{P}^2)=1$, and
Constructions~I, II and~III preserve this equality.
At the end of the day we arrive to smooth varieties $V$ and~$U$,
and thus conlcude that
$$
\mathrm{rk~Pic}(V/U)=\mathrm{rk~Cl}(V/U)=1.
$$
This means that $\nu$ is a standard conic bundle.
Since the divisor $-K_V$ is $\nu$-ample and $U$ is a projective surface,
we see that $V$ is also projective (although a result of a flop may a priori be not projective).

Other assertions of the theorem hold by construction.
\end{proof}

Constructions~I, II, and~III are analogues of the constructions in the proof of \cite[Proposition~2.4]{Sarkisov1}.

\medskip

Theorem~\ref{theorem:conic-bundle-standard-weak-del-Pezzo} gives the following.

\begin{corollary}
\label{corollary:Shokurov-factorial}
Conjecture~\ref{conjecture:Shokurov} implies Conjecture~\ref{conjecture:factorial}.
\end{corollary}

\begin{proof}
We have $2K_U+\Delta\sim 0$ by Theorem~\ref{theorem:conic-bundle-standard-weak-del-Pezzo}(iii),
so that the linear system $|2K_U+\Delta|$ is not empty.
Thus, the irrationality of $X$ follows from Conjecture~\ref{conjecture:Shokurov}.
\end{proof}

\section{Irrational quartic double solids}
\label{section:factorial}

The main goal of this section is to prove Theorem~\ref{theorem:CPS}.
To achieve it,
we need the following straightforward result.

\begin{proposition}
\label{proposition:blow-up-plane}
Let $U$ be a smooth surface, and let $\Delta$ be a reduced curve in $U$.
Suppose that
$$
\Delta\sim -2K_U,
$$
the curve $\Delta$ is not a smooth rational curve, and $\Delta$ satisfies
conditions $(\mathrm{A})$ and $(\mathrm{B})$ in Corollary~\ref{corollary:stability}.
Then $-K_U$ is numerically effective (nef).
\end{proposition}

\begin{proof}
Suppose that $-K_U$ is not nef.
Then there is an irreducible curve $\Delta_1\subset U$
such that~\mbox{$\Delta\cdot \Delta_1<0$}.
This, in particular, means that $\Delta_1$ is an irreducible component of $\Delta$.

We claim that $\Delta$ is a reducible curve. Indeed, if $\Delta=\Delta_1$, then $\Delta^2<0$.
On the other hand, the adjunction formula implies that
$$
2p_a(\Delta)-2=\Delta^2+K_U\cdot \Delta=\Delta^2-\frac{\Delta^2}{2}=\frac{\Delta^2}{2},
$$
where $p_a(\Delta)=1-\chi(\mathcal{O}_{\Delta})$ is the arithmetic genus of $\Delta$.
Thus, if $\Delta$ is irreducible, then its arithmetic genus must be zero,
so that $\Delta$ is a smooth rational curve.
The latter is impossible, because $\Delta$ satisfies
condition $(\mathrm{B})$ of Corollary~\ref{corollary:stability}.

We see that $\Delta$ is reducible, and $\Delta_1$ is its irreducible component.
Denote the union of its remaining irreducible components by $\Delta_2$,
so that $\Delta=\Delta_1\cup\Delta_2$. Then
$$
0>\Delta\cdot \Delta_1=\Delta_1^2+\Delta_1\cdot \Delta_2.
$$
On the other hand, the adjunction formula gives
$$
2p_a(\Delta_1)-2=\Delta_1^2+K_U\cdot \Delta_1=\Delta_1^2-\left(\frac{\Delta_1+\Delta_2}{2}\right)\cdot \Delta_1=\frac{\Delta_1^2}{2}-\frac{\Delta_1\cdot \Delta_2}{2},
$$
so that
$$4p_a(\Delta_1)-4=\Delta_1^2-\Delta_1\cdot \Delta_2.$$
Thus, if $p_a(\Delta_1)>0$, then
$$
0\leqslant 4p_a(\Delta_1)-4=\Delta_1^2-\Delta_1\cdot \Delta_2<-2\Delta_1\cdot \Delta_2,
$$
which gives a contradiction with $\Delta_1\cdot \Delta_2\geqslant 0$.
Hence, we have
$p_a(\Delta_1)=0$, which implies that $\Delta_1$ is a smooth rational curve.
Therefore
$$-4=\Delta_1^2-\Delta_1\cdot \Delta_2$$
by the adjunction formula.
Since $\Delta_1^2+\Delta_1\cdot \Delta_2<0$ by assumption, we have $\Delta_1^2<-2$.
Thus
$$
-4=\Delta_1^2-\Delta_1\cdot \Delta_2<-2-\Delta_1\cdot \Delta_2,
$$
which gives $\Delta_1\cdot \Delta_2\leqslant 1$.
This is impossible, because  $\Delta$ satisfies
conditions $(\mathrm{A})$ and $(\mathrm{B})$
of Corollary~\ref{corollary:stability}.
\end{proof}

Now let $\tau\colon X\to\mathbb{P}^3$ be a double cover branched over a nodal quartic surface $S$.
To prove Theorem~\ref{theorem:CPS},
we must prove that $X$ is irrational
when $S$ has at most six nodes.

If $S$ is smooth, then $X$ is irrational by \cite[Corollary~4.7(b)]{Voisin88}.
Thus, we may assume that $S$ is singular.
If  $S$ has at most five nodes, then $X$ is  $\mathbb{Q}$-factorial by Theorem~\ref{theorem:Cheltsov}.
Similarly, if $S$ has exactly six nodes, then $X$ is  $\mathbb{Q}$-factorial by Theorem~\ref{theorem:Cheltsov}
with the only exception when $X$ is birational to a smooth cubic
threefold in $\mathbb{P}^4$, and thus irrational by~\cite[Theorem~13.12]{ClemensGriffiths}.
Therefore, we may also assume that $X$ is $\mathbb{Q}$-factorial.
Thus, we can apply all results of Sections~\ref{section:double-solids} and \ref{section:standard} to $X$.

By Theorem~\ref{theorem:conic-bundle-standard-weak-del-Pezzo},
the threefold $X$ is birational to a smooth threefold $V$ with a structure
of a standard conic bundle $\nu\colon V\to U$ and there exists a birational morphism
$\varrho\colon U\to\mathbb{P}^2$
that is a composition of $|\mathrm{Sing}(S)|-1$ blow ups.
Denote by $\Delta$ the degeneration curve of the conic bundle $\nu$.
In particular, there exists a pair $(\Delta^\prime,I)$ of a connected nodal curve~$\Delta^\prime$ and an involution $I$ on it
such that $\Delta\cong\Delta^\prime/I$, the nodes of $\Delta^\prime$ are exactly the fixed points of $I$,
and $I$ does not interchange branches at these points.
One has $\Delta\sim -2K_U$ by Theorem~\ref{theorem:conic-bundle-standard-weak-del-Pezzo}(iii).

By Corollary~\ref{corollary:stability},
the curve $\Delta$ satisfies its conditions
$(\mathrm{A})$ and $(\mathrm{B})$.
Thus, $-K_U$ is nef by Proposition~\ref{proposition:blow-up-plane}.
Put $d=K_U^2$. Then
$$
d=10-\big|\mathrm{Sing}(S)\big|
$$
by Theorem~\ref{theorem:conic-bundle-standard-weak-del-Pezzo}(i).

\medskip

Now we are ready to prove Theorem~\ref{theorem:CPS}.
Until the end of the section we assume
that~$S$ has at most six nodes, so that $d\geqslant 4$.
In particular, $U$ is a \emph{weak del Pezzo surface} (see~\cite{Demazure}).

\begin{lemma}
\label{lemma:weak-Del-Pezzo-Delta-connected}
The curve $\Delta$ is connected.
\end{lemma}

\begin{proof}
Since $-K_U$ is nef and big, we have $h^1(\mathcal{O}_U(-2K_U))=0$ by the Kawamata--Viehweg vanishing theorem (see \cite{Kawamata-vanishing}).
This implies connectedness of $\Delta$, because $\Delta\sim -2K_U$.
\end{proof}

We plan to apply Theorem~\ref{theorem:Shokurov} to $V$.
Unfortunately, the curve $\Delta$ may not satisfy
condition~$(\mathrm{S})$.
Luckily, we can explicitly describe each case when $\Delta$
does not satisfy it.
This description is given by the following three lemmas.

\begin{lemma}
\label{lemma:weak-Del-Pezzo-very-easy}
Let $E$ be a $(-2)$-curve on $U$.
Then either $E$ is contracted by $\varrho$ to a point,
or $\varrho(E)$ is a line in~$\P^2$.
Moreover, either $E$ is disjoint from $\Delta$, or $E$ is an irreducible component of $\Delta$.
Furthermore, if $E$ is an irreducible component of $\Delta$,
then it intersects the curve~\mbox{$\Delta-E$} by two points.
In particular, if $\Delta$ satisfies condition  $(\mathrm{S})$
of Theorem~\ref{theorem:Shokurov}, then $E$ is disjoint from $\Delta$.
\end{lemma}

\begin{proof}
If $E$ is not contracted by $\varrho$ to a point, then $\varrho(E)$ is a line, because $d\geqslant 4$.
If $E$ is not an irreducible component of the curve $\Delta$, then
$$
\Delta\cdot E=-2K_U\cdot E=0,
$$
which implies that $E$ is disjoint from $\Delta$.
If $E$ is an irreducible component of $\Delta$, then
$$
(\Delta-E)\cdot E=(-2K_U-E)\cdot E=2.
$$
Since $\Delta$ is nodal, this means that $\Delta-E$ intersects $E$ by two points.
In particular, $\Delta$ does not satisfy condition~$(\mathrm{S})$
of Theorem~\ref{theorem:Shokurov} in this case.
\end{proof}

\begin{lemma}
\label{lemma:weak-Del-Pezzo-easy-1}
At most two irreducible components of the curve $\Delta$ are $(-2)$-curves.
Moreover, all other $(-2)$-curves on $U$ are disjoint from $\Delta$.
Furthermore, for the curve $\Delta$ we have only the following possibilities:
\begin{itemize}
\item the curve $\Delta$ contains a unique $(-2)$-curve $E$ in $U$ and
$$
\Delta=E+\Omega,
$$
where $\Omega$ is a nodal curve such that $\varrho(\Omega)$ is a (possibly reducible) quintic curve,~$\varrho(E)$ is a line, and $E\cap\Omega$ consists of two points;

\item the curve $\Delta$ contains two $(-2)$-curves $E_1$ and $E_2$ in $U$,
the curves $E_1$ and $E_2$ are disjoint, and
$$
\Delta=E_1+E_2+\Upsilon,
$$
where $\Upsilon$ is a nodal curve such that $\varrho(\Upsilon)$ is a (possibly reducible) quartic curve,~$\varrho(E_1)$ and~$\varrho(E_1)$  are lines, and each intersection $E_1\cap\Upsilon$ or $E_1\cap\Upsilon$ consists of two points.
\end{itemize}
\end{lemma}

\begin{proof}
Suppose that at least three irreducible components of $\Delta$ are  $(-2)$-curves.
Denote them by $E_1$, $E_2$, and $E_3$.
Then $\varrho$ maps them to lines in $\mathbb{P}^2$ by Lemma~\ref{lemma:weak-Del-Pezzo-very-easy}.
This implies that $\varrho$ blows up at least six points on these lines,
which is impossible, because we assume that $d\geqslant 4$.
This shows that at most two irreducible components of the curve $\Delta$ are $(-2)$-curves.
All remaining assertions easily follow from Lemma~\ref{lemma:weak-Del-Pezzo-very-easy}.
\end{proof}

\begin{lemma}
\label{lemma:weak-Del-Pezzo-easy-3}
Suppose that $\Delta$ does not satisfy condition $(\mathrm{S})$ of Theorem~\ref{theorem:Shokurov},
i.~e. there is a splitting $\Delta=\Delta_1\cup\Delta_2$ such that $\Delta_1\cdot\Delta_2=2$.
Then either $\Delta_1$ or $\Delta_2$ is a $(-2)$-curve.
\end{lemma}

\begin{proof}
We claim that $\Delta_1$ and $\Delta_2$ are linearly
independent in $\mathrm{Pic}(S)\otimes\mathbb{Q}$.
Indeed, let~\mbox{$\Delta_1\sim_{\mathbb{Q}} \lambda \Delta_2$}
for some rational number $\lambda$. Interchanging $\Delta_1$ and $\Delta_2$ if needed one can assume that $\lambda\geq 1$.
One has
$$
2=\Delta_1\cdot\Delta_2=\lambda \Delta_2^2.
$$
Since $\Delta_2^2\geqslant 1$, one has
$$
2=\lambda \Delta_2^2\geqslant\lambda\geqslant 1.
$$
Moreover, one has
$$-K_U\sim_{\mathbb{Q}}\frac{\lambda+1}{2}\Delta_2$$
and
$$
4\leqslant \left(-K_U\right)^2=\left(\frac{\lambda+1}{2}\right)^2\Delta_2^2=\frac{2}{\lambda}\left(\frac{\lambda+1}{2}\right)^2,
$$
which is impossible for $1\leqslant \lambda\leqslant 2$.

Applying the Hodge index Theorem, we get
$$\left|
    \begin{array}{cc}
      \Delta_1^2 & \Delta_1\Delta_2 \\
      \Delta_1\Delta_2 & \Delta_1^2 \\
    \end{array}
  \right|<0,
$$
which means that
\begin{equation}\label{eq:Hodge}
\Delta_1^2\Delta_2^2<\left(\Delta_1\Delta_2\right)^2=4.
\end{equation}

We claim that an arithmetic genus of either
$\Delta_1$ or $\Delta_2$ is non-positive.
Indeed, otherwise
$$
0\leqslant 2p_a(\Delta_i)-2=\Delta_i(\Delta_i+K_U)=\frac{\Delta_i^2}{2}-\frac{\Delta_1\Delta_2}{2}=\frac{\Delta_i^2}{2}-1.
$$
This means that $\Delta_i^2\geqslant 2$, and thus
$\Delta_1^2\Delta_2^2\geqslant 4$, which contradicts~\eqref{eq:Hodge}.

Without loss of generality, we may assume that $p_a(\Delta_1)\leqslant 0$.
Then $\Delta_1^2\leqslant -2$ by the adjunction formula.
Thus, if $\Delta_1$ is irreducible, then we are done.
Therefore, we assume that $\Delta_1$ has at least two irreducible components.
Then the degree of the curve
$\varrho(\Delta_1)$ is at least two, and thus the degree of the curve
$\varrho(\Delta_2)$ is at most four.
On the other hand, we have
$$
\Delta_2^2\geqslant \Delta_2^2+\Delta_1^2+2=
\left(\Delta_1+\Delta_2\right)^2-2=
\left(-2K_U\right)^2-2=4d-2\geqslant 14,
$$
which implies that the degree of the curve
$\varrho(\Delta_2)$ is exactly four,
$\Delta_1$ has exactly two irreducible components,
and each of these components is mapped by $\varrho$ to a line in $\mathbb{P}^2$.
This is impossible, because
$$
2=\Delta_1\cdot\Delta_2\geqslant\varrho(\Delta_1)\cdot \varrho(\Delta_2)-(9-d)=d-1\geqslant 3,
$$
which is absurd.
\end{proof}

Since $d\geqslant 4$, the linear system $|-K_U|$ is base point free and gives a morphism $\phi\colon U\to\mathbb{P}^d$.
Denote by $Y$ the image of $U$ via $\phi$.
Then $U$ is a del Pezzo surface with du Val singularities,
and $\phi$ induces a birational morphism $\varphi\colon U\to Y$ that contracts all $(-2)$-curves on $U$
to singular points of the surface $Y$.
Since $d\geqslant 4$, the surface $Y$ is an intersection of quadrics in $\mathbb{P}^d$, see, for example, \cite{Demazure}.

Put $\nabla=\varphi(\Delta)$. Then $\nabla\in |-2K_{Y}|$.
By Lemma~\ref{lemma:weak-Del-Pezzo-easy-1},
the curve $\nabla$ is connected and nodal.
Moreover, it follows from Lemma~\ref{lemma:weak-Del-Pezzo-easy-1} and Remark~\ref{remark:Shokurov}
that there exists a pair~\mbox{$(\nabla^\prime,J)$}
of a connected nodal curve $\nabla^\prime$ and an involution $J$ on it
such that $\nabla\cong\nabla^\prime/J$, the nodes of $\nabla^\prime$ are exactly the fixed points of $J$,
the involution $J$ does not interchange branches at these points, and
$$
\mathrm{Prym}\left(\Delta^\prime,I\right)\cong\mathrm{Prym}\left(\nabla^\prime,J\right).
$$

Since $U$ is rational, the exact sequence
$$
0\to \mathcal O_U\to \omega_U\otimes \mathcal O_U(\Delta)\to \omega_\Delta\to 0
$$
implies that there is a surjection
$$H^0(-K_U)=H^0(K_U+\Delta)\twoheadrightarrow H^0(K_\Delta),$$
so for the anticanonical map $\phi$ one has
$$
\left.\phi\right|_{\Delta}=\kappa_\Delta,
$$
where $\kappa_\Delta$ is a canonical map of the curve $\Delta$.
Therefore, we see that $\nabla$ is a connected nodal curve
canonically embedded into $\mathbb{P}^d$, which is an intersection of quadrics.
In particular,~$\nabla$ is not  hyperelliptic by Remark~\ref{remark:hypertrig}.
Moreover, for every splitting $\nabla=\nabla_1\cup\nabla_2$,
the intersection $\nabla_1\cap\nabla_2$ consists of at least $4$ points.
This follows from Corollary~\ref{corollary:stability} and Lemmas~\ref{lemma:weak-Del-Pezzo-easy-1} and \ref{lemma:weak-Del-Pezzo-easy-3}.
Thus, $\mathrm{Prym}(\nabla^\prime,J)$ is not a sum of Jacobians of smooth
curves by Corollary~\ref{corollary:Shokurov}.
On the other hand, Theorem~\ref{theorem:Jac-Prym} implies that
$$
\mathrm{J}(V)\cong\mathrm{Prym}\left(\Delta^\prime,I\right)\cong\mathrm{Prym}\left(\nabla^\prime,J\right),
$$
where $\mathrm{J}(V)$
is the intermediate Jacobian of the threefold $V$.
This shows that $V$ is irrational by \cite[Corollary~3.26]{ClemensGriffiths}
and completes the proof of Theorem~\ref{theorem:CPS}.

\begin{question}
There are three other smooth double covers of smooth varieties among Fano threefolds of Picard rank one: a sextic double solid, a double quadric, and some special variety of type $V_{10}$. The first two are irrational, and a general double cover of type~$V_{10}$ is irrational as well (see, for instance,~\cite[\S12.2]{IsPr99}).
Some of nodal varieties of these three types are proven to be irrational
as well (see, for instance, \cite{CheltsovPark}, \cite[Proposition~3.1]{IKP}, and~\cite{Bea16}
for the sextic double solid, and~\cite{Grin},~\cite{Shramov}, and~\cite{PSh} for the double quadric, and).
In particular, it follows from \cite{CheltsovPark} and \cite{Shramov}
that neither a $\mathbb{Q}$-factorial nodal
sextic double solid, nor a $\mathbb{Q}$-factorial nodal double quadric can be birational
to a standard conic bundle. Is it possible to apply the techniques used for the proof of Theorem~\ref{theorem:CPS} for nodal double covers of type~$V_{10}$?
\end{question}

\section{Rational quartic double solids}
\label{section:non-factorial}

In this section, we prove Theorem~\ref{theorem:non-factorial}.
Let $\tau\colon X\to\mathbb{P}^3$ be a double cover
branched over a nodal quartic surface $S$.
Suppose that $X$ is not $\mathbb{Q}$-factorial. We are going to
show that~$X$ is rational unless it is
described by Example~\ref{example:cubic-threefold}.

Let $f\colon{X}^\prime\to X$ be a $\mathbb{Q}$-factorialization of
the threefold $X$ (see \cite[Corollary~4.5]{Kawamata}), so that
the threefold $X^\prime$ has $\mathbb{Q}$-factorial terminal singularities.
Then
\begin{equation}
\label{equation:anticanclass}
-K_{{X}^\prime}\sim f^{*}\left(\tau^*\left(\mathcal{O}_{\mathbb{P}^3}(2)\right)\right).
\end{equation}
Moreover, since in a neighborhood of every node of $X$
the $\mathbb{Q}$-factorialization $f$ is either an isomorphism or a small resolution
of the node, we conclude that $X^\prime$ is also has at most nodes as singularities.
In particular, $X^\prime$ is Gorenstein.

Since $K_{{X}^\prime}$ is not nef, the cone
$\overline{\mathrm{NE}}({X}^\prime)$ has an extremal
ray that has negative intersection with $K_{{X}^\prime}$.
Let $\eta\colon{X}^\prime\to Y$ be a contraction of this extremal ray (see \cite[Corollary~2.9]{Shokurov2}).
Since $X$ is not $\mathbb{Q}$-factorial,
the rank of the Picard group of $X^\prime$ is at least $2$.
This means that the rank of the Picard group of $X^\prime$ is at least $1$,
so that $Y$ is not a point.

If $\eta$ is a conic bundle, then \eqref{equation:anticanclass} implies that
the pull-back of a plane in $\P^3$ via $\tau\circ f$ is a section of $\eta$,
so that $X$ is rational.
Similarly, if $\eta$ is a del Pezzo fibration, then
\eqref{equation:anticanclass} implies that the canonical class of its
general fiber is divisible by two in the Picard group,
so that the general fiber of $\eta$ is a quadric surface
by the adjunction formula, and $X$ is rational in this case as well.

Hence, to complete the proof, we may assume that $\eta$ is birational.
Note that $f$ cannot be a small contraction because $X^\prime$ is Gorenstein,
see~\cite[Theorem~6.2]{Mori}. If
$f$ contracts some surface $E$ to a curve, then a fiber of
$f$ over a general point of $E$ is a curve $L$ with~\mbox{$K_{X^\prime}\cdot L=-1$};
this is impossible by~\eqref{equation:anticanclass}. Thus
$f$ contracts some divisor $E$ to a point. It follows from~\cite[Theorem~5]{Cutkosky} that $f$ is a blow up of a smooth point in $Y$. Indeed, if $f$ is like in one of cases~(2), (3), or~(4) in the notation
of \cite[Theorem~5]{Cutkosky}, then it is easy to produce a curve $L$ contained in $E$ with odd
intersection $K_{X^\prime}\cdot L$; the latter is forbidden by~\eqref{equation:anticanclass}.
Hence $f$ is described by case~(1) of~\cite[Theorem~5]{Cutkosky}, which means that $P=f(E)$
is a smooth point on $Y$, and $f$ is the blow up of~$P$.
Let us denote this point by $P$.

The divisor $-K_{Y}$ is nef by \eqref{equation:anticanclass}. Moreover, we have
$$
\left(-K_{Y}\right)^3=\left(-K_{X^\prime}\right)^3+8=
\left(-K_{X}\right)^3+8=24,
$$
which implies that $-K_{Y}$ is big.
By \cite[Theorem~2.1]{Shokurov2}, the linear system $|-nK_Y|$ is base point free for some $n>0$,
and it gives a birational morphism $\phi\colon Y\to Z$ such that $Z$ is a Fano threefold with canonical Gorenstein singularities.
Moreover, \cite[Theorem~2.1]{Shokurov2} also implies that $-K_{Z}$ is divisible by $2$ in $\mathrm{Pic}(Z)$.
Since
$$
\left(-K_{Z}\right)^3=\left(-K_{Y}\right)^3=24,
$$
the threefold $Z$ must be isomorphic to a cubic threefold
in $\mathbb{P}^4$ (see, for example,~\mbox{\cite[Theorem~3.4]{Prokhorov}}).
In particular, if $Z$ is singular, then $X$ is rational.
Thus, to complete the proof, we may assume that $Z$ is smooth.
Then $\phi$ is an isomorphism, so we may assume that~\mbox{$Z=Y$}.
Therefore, there is a commutative diagram
$$
\xymatrix{
{X}^\prime
\ar@{->}[rr]^{\eta}\ar@{->}[d]_{f}&&Y\ar@{-->}[d]^{\gamma}\\
X\ar@{->}[rr]^{\tau}&&\mathbb{P}^{3}}
$$
where $\gamma$ is a linear projection from the point $P$.
Since $X$ is nodal, the cubic $Y$ contains exactly six lines that pass through $P$,
and $f$ is the contraction of their proper transforms.
This means that
$X$ is described by Example~\ref{example:cubic-threefold}.

\end{document}